\documentclass[12pt]{amsart}

\usepackage{amsmath,amsthm,amssymb}
\usepackage{amsfonts}
\usepackage[mathscr]{eucal}
\usepackage{ascmac}
\usepackage{epic, eepic}
\usepackage{graphicx}
\usepackage{cases}
\usepackage{color}
\usepackage{url}
\usepackage{bm}
\usepackage[all]{xy}
\usepackage[top=25truemm,bottom=25truemm,left=35truemm,right=35truemm]{geometry}
\usepackage[noadjust]{cite}
\usepackage{enumerate}
\theoremstyle{plane}
\newtheorem{thm}{Theorem}[section]
\newtheorem{prop}[thm]{Proposition}%[section]
\newtheorem{lem}[thm]{Lemma}%[section]
\newtheorem{cor}[thm]{Corollary}%[section]
\newtheorem{fact}[thm]{Fact}%[section]
\theoremstyle{definition}
\newtheorem{dfn}[thm]{Definition}%[section]
\newtheorem{exa}[thm]{Example}%[section]
\theoremstyle{remark}
\newcommand{\rank}{\operatorname{rank}}

\newcommand{\im}{\operatorname{Im}}
\newcommand{\hess}{\operatorname{Hess}}

\newcommand{\sgn}{\operatorname{sgn}}

\newcommand{\R}{\bm{R}}

\newcommand{\Sig}{\Sigma}
\numberwithin{equation}{section}
\renewcommand{\phi}{\varphi}
\renewcommand{\epsilon}{\varepsilon}
\newcommand{\what}{\widehat}
\newcommand{\til}{\tilde}
\newcommand{\wtil}{\widetilde}
\newcommand{\inner}[2]{\left\langle{#1},{#2}\right\rangle}

\begin{document}
%%%%% To ease editing, for IMPAN journals add:

%\baselineskip=17pt

%%%%%%%%%%%
\title[Gauss maps of fronts]{Singularities of Gauss maps of wave fronts 
with non-degenerate singular points}
\author[K. Teramoto]{Keisuke Teramoto}
\thanks{The author was supported by the Grant-in-Aid for JSPS Fellows Number 17J02151.}
\address{Department of Mathematics, %Graduate School of Science, 
Kobe University, 
Rokko 1-1, Nada, Kobe 657-8501, Japan}
\email{teramoto@math.kobe-u.ac.jp}

\maketitle
\begin{abstract}
We study singularities of Gauss maps of fronts and give characterizations of types of singularities of Gauss maps 
by geometric properties of fronts which are related to behavior of bounded principal curvatures. 
Moreover, we investigate relation between 
a kind of boundedness of Gaussian curvatures near cuspidal edges and types of singularities of Gauss maps 
of cuspidal edges. 
Further, we consider extended height functions on fronts with non-degenerate singular points. 
\end{abstract}

\section{Introduction}
In study of (extrinsic) differential geometry of surfaces in the Euclidean $3$-space $\R^3$, 
Gauss maps play important roles. 
For instance, singular points of the Gauss map coincide with the parabolic points,
namely, points at which the Gaussian curvature vanishes. 
On the other hand, it is known that height functions in the normal direction of surfaces have singularities. 
Height functions on surfaces measure types of contact of surfaces and planes.  
Generically, height functions on surfaces have $A_k$ singularities (\cite{ifrt}). 
In particular, height functions on surfaces have $D_4$ singularities at {flat umbilic points} (cf. \cite{bgt2}).

In this paper, we study singularities of Gauss maps of (\/wave\/) fronts and height functions. 
In general, the Gaussian curvature of a front is unbounded near singular points. 
However, the Gaussian curvature of a front is {\it rationally  bounded}, which is a kind of boundedness 
introduced by Martins, Saji, Umehara and Yamada \cite{msuy}, at a singular point 
which is also a singular point of the Gauss map. 
Therefore studying singularities of Gauss maps of fronts might be related to 
investigate the behavior of the Gaussian curvature of a front near a singular point. 
In fact, we will show relationships between rational boundedness and 
contact order of a singular curve and a parabolic curve for a cuspidal edge. 
To do this, we need to consider the behavior of a bounded principal curvature near 
a singular point of a front. 
It is known that a bounded principal curvature of a front coincides with 
the {\it limiting normal curvature} $\kappa_\nu$ at a non-degenerate singular point, 
which contains a cuspidal edge or a swallowtail. 
Moreover, a non-degenerate singular point is also a singular point of 
the corresponding Gauss map of a front if and only if $\kappa_\nu$ vanishes at this point. 
Thus, to study types of singularities of a Gauss map, 
we consider properties of a bounded principal curvature of a front at non-degenerate 
singular point (Theorem \ref{typegauss}). 
Further, we consider contact between the singular curve and the singular set of the Gauss map for a cuspidal edge. 
We give some geometric interpretations of a rational boundedness and a rational continuity 
of the Gaussian curvature of a cuspidal edge in terms of contact properties of two curves 
(Corollaries \ref{boundedness} and \ref{singbound}).

In addition, we study height functions on fronts with non-degenerate singular points of the second kind, 
which do not contain cuspidal edges. 
In particular, we consider contact between fronts with non-degenerate singular points of the second kind 
and their {\it limiting tangent planes} and show a condition that corresponding height functions 
have $D_4$ singularities in terms of differential geometric properties of initial fronts (Theorem \ref{D4phi}).

\section{Preliminaries}\label{sec:1}
\subsection{Fronts}\label{sec:2}
We review some notions of (wave) fronts. 
For details, see \cite{agv,ifrt,krsuy,msuy,suy2,suy}. 

Let $f:\Sig\to\R^3$ be a $C^\infty$ map, where $\Sig\subset(\R^2;u,v)$ is a domain. 
We then say that $f$ is a {\it frontal} if there exists a unit vector field $\nu$ along $f$ such that 
$\langle df_q(X_q),\nu(q)\rangle=0$ for any $q\in \Sig$ and $X_q\in T_q\Sig$, 
where $\langle\cdot,\cdot\rangle$ is a usual inner product of $\R^3$. 
This vector field $\nu$ is called a {\it unit normal vector} or the {\it Gauss map} of $f$. 
A frontal $f$ is called a ({\it wave}) {\it front} if the pair of mappings 
$L_f=(f,\nu):\Sig\to\R^3\times S^2$ gives an immersion, where $S^2$ denotes the unit sphere in $\R^3$. 

We fix a frontal $f$. 
A point $p\in\Sig$ is a {\it singular point} if $\rank df_p<2$ holds. 
We denote by $S(f)$ the set of singular points of $f$ 
and call the image $f(S(f))$ of $S(f)$ by $f$ the {\it singular locus}. 
We set a function $\lambda:\Sig\to\R$ by 
$$\lambda(u,v)=\det(f_u,f_v,\nu)(u,v),$$
where $\det$ is the usual determinant of $3\times3$ matrices, 
$f_u=\partial f/\partial u$ and $f_v=\partial f/\partial v$. 
We call this function $\lambda$ the {\it signed area density function}. 
It is obviously $S(f)=\lambda^{-1}(0)$. 
Take a singular point $p\in S(f)$. 
Then $p$ is {\it non-degenerate} if $d\lambda(p)\neq0$, that is, $(\lambda_u(p),\lambda_v(p))\neq(0,0)$. 
If $p$ is a non-degenerate singular point of $f$, 
then there exist a neighborhood $V\subset\Sig$ of $p$ and 
a regular curve $\gamma:(-\epsilon,\epsilon)\ni t\mapsto\gamma(t)\in V$ $(\epsilon>0)$ 
with $\gamma(0)=p$ such that $\im(\gamma)=S(f)\cap V$ by the implicit function theorem, 
where $\im(\gamma)$ is the image of $\gamma$. 
Moreover, there exists a never vanishing vector field $\eta$ on $S(f)\cap V$ 
such that for any $q\in S(f)\cap V$, $df_q(\eta_q)=0$ holds. 
We call $\gamma$, $\hat{\gamma}=f\circ\gamma$ and $\eta$ 
a {\it singular curve}, a {\it singular locus} and a {\it null vector field}, respectively.
If a vector field $\tilde{\eta}$ on $V$ satisfies $\tilde{\eta}|_{S(f)\cap V}=\eta$, 
then $\tilde{\eta}$ is called an {\it extended null vector field} (\cite{suy3,suy}). 

A non-degenerate singular point $p$ is called the {\it first kind} if $\det(\gamma',\eta)(0)\neq0$, 
where $'=d/dt$. 
Otherwise, $p$ is called the {\it second kind} (\cite{msuy}). 

\begin{dfn}
Let $f,g:(\R^2,0)\to(\R^3,0)$ be $C^\infty$ map germs. 
Then $f$ and $g$ are {\it $\mathcal{A}$-equivalent} if there exist diffeomorphism germs 
$\theta:(\R^2,0)\to(\R^2,0)$ and $\Theta:(\R^3,0)\to(\R^3,0)$ such that 
$\Theta\circ f=g\circ\theta$ holds.
\end{dfn}
\begin{dfn}
Let $f:(\Sig,p)\to(\R^3,f(p))$ be a map germ around $p$. 
Then $f$ at $p$ is a \textit{cuspidal edge} if the map germ $f$ is $\mathcal{A}$-equivalent to the map germ 
$(u,v)\mapsto(u,v^2,v^3)$ at ${0}$, 
$f$ at $p$ is a \textit{swallowtail} if the map germ $f$ is $\mathcal{A}$-equivalent to the map germ 
$(u,v)\mapsto(u,3v^4+uv^2,4v^3+2uv)$ at ${0}$.% (see Figure \ref{genefront}). 
\end{dfn}

These singularities are {\it generic singularities} of fronts in $\R^3$ (\cite{agv}). 
Moreover, a cuspidal edge is non-degenerate singular point of the first kind, and 
a swallowtail is of the second kind (\cite{msuy}). 

Criteria for these singularities are known.
\begin{fact}[\cite{krsuy,suy3}]\label{crit:non-deg}
Let $f:(\Sig,p)\to(\R^3,f(p))$ be a front germ and $p\in\Sig$ a singular point of $f$. 
Then the following assertions hold:
\begin{enumerate}
\item $f$ at $p$ is $\mathcal{A}$-equivalent to the cuspidal edge if and only if $\eta\lambda(p)\neq0$.
\item $f$ at $p$ is $\mathcal{A}$-equivalent to the swallowtail if and only if $d\lambda(p)\neq0$, 
$\eta\lambda(p)=0$ and $\eta\eta\lambda(p)\neq0$.
\end{enumerate}
\end{fact}
We note that criteria for other singularities of fronts and frontals are known 
(\cite{fsuy,is,ist,s}).

\subsection{Geometric invariants of fronts}\label{sec:3}
Let $f:\Sig\to\R^3$ be a front, $\nu$ its unit normal vector and $p\in\Sig$ a non-degenerate singular point of $f$. 
Then we can take the following local coordinate system $(U;u,v)$ centered at $p$.
\begin{dfn}
Let $f:\Sig\to\R^3$ be a front and $p$ a cuspidal edge (resp. non-degenerate singular point of the second kind). 
Then a local coordinate system $(U;u,v)$ centered at $p$ is called an {\it adapted} 
if the following conditions hold:
\begin{itemize}
\item the $u$-axis gives a singular curve,
\item $\eta=\partial_v$ (resp. $\eta=\partial_u+\epsilon(u)\partial_v$ with $\epsilon(0)=0$) 
gives a null vector field on the $u$-axis, and
\item there are no singular point other than the $u$-axis.
\end{itemize}
\end{dfn}

First we deal with cuspidal edges. 
Let $f:\Sig\to\R^3$ be a front, $\nu$ its unit normal vector and 
$p\in\Sig$ a cuspidal edge. 
Let $\kappa_s$, $\kappa_\nu$, $\kappa_c$, $\kappa_t$ and $\kappa_i$ denote the {\it singular curvature} (\cite{suy}), 
the {\it limiting normal curvature} (\cite{suy}), the {\it cuspidal curvature} (\cite{msuy}), 
the {\it cusp-directional torsion} (\cite{ms}) and the {\it edge inflectional curvature} (\cite{msuy}), respectively. 
{If we take an adapted coordinate system $(U;u,v)$ around $p$, then 
\begin{align}\label{kappas}
\begin{aligned}
\kappa_s&=\sgn(\lambda_v)\dfrac{\det(f_u,f_{uu},\nu)}{|f_u|^3},\quad
\kappa_\nu=\dfrac{\langle{f_{uu},\nu}\rangle}{|f_u|^2},\quad
\kappa_c=\dfrac{|f_u|^{3/2}\det(f_u,f_{vv},f_{vvv})}{|f_u\times f_{vv}|^{5/2}},\\
\kappa_t&=\dfrac{\det(f_u,f_{vv},f_{uvv})}{|f_u\times f_{vv}|^2}
-\dfrac{\det(f_u,f_{vv},f_{uu})\langle{f_u,f_{vv}}\rangle}{|f_u|^2|f_u\times f_{vv}|^2},\\
\kappa_i&=\dfrac{\det(f_u,f_{vv},f_{uuu})}{|f_u|^3|f_u\times f_{vv}|}-
3\dfrac{\langle{f_u,f_{vv}}\rangle\det(f_u,f_{vv},f_{uu})}{|f_u|^5|f_u\times f_{vv}|}
\end{aligned}
\end{align}
hold on the $u$-axis, where $|\cdot|$ denotes the standard norm of $\R^3$.} 
We note that $\kappa_c$ does not vanish along the singular curve 
if it consists of cuspidal edges (\cite[Proposition 3.11]{msuy}). 
Moreover, $\kappa_s$ is an intrinsic invariant of a cuspidal edge 
which relates to the convexity and concavity (\cite{hhnsuy,suy}). 
For other geometric meanings of these invariants, see \cite{hhnsuy,hnuy,suy,msuy,ms}. 

Take an adapted coordinate system $(U;u,v)$ centered at $p$. 
Since $df(\eta)=f_v=0$ on the $u$-axis, there exists a $C^\infty$ map 
$h:U\to\R^3\setminus\{0\}$ such that $f_v=vh$ on $U$. 
We note that $\{f_u,h,\nu\}$ gives a frame on $U$, 
and we may take $\nu$ as $\nu=(f_u\times h)/|f_u\times h|$ (cf. \cite{ms,msuy,nuy}). 
Under the adapted coordinate system $(U;u,v)$ centered at $p$ with $\lambda_v(u,0)=\det(f_u,h,\nu)(u,0)>0$, 
invariants $\kappa_s$, $\kappa_\nu$, $\kappa_c$ and $\kappa_t$ as in \eqref{kappas} 
can be written as follows (\cite[(2.17)]{hhnsuy} and \cite[Lemma 2.7]{t2}): 
\begin{align}\label{inv_cusp}
\begin{aligned}
\kappa_s&=\frac{2\wtil{F}_u\wtil{E}-\wtil{E}\wtil{E}_{vv}-\wtil{E}_u\wtil{F}}
{2\wtil{E}^{3/2}\sqrt{\wtil{E}\wtil{G}-\wtil{F}^2}},\quad
\kappa_\nu=\frac{\wtil{L}}{\wtil{E}},\\
\kappa_c&=2\wtil{N}\left(\dfrac{\wtil{E}}{\wtil{E}\wtil{G}-\wtil{F}^2}\right)^{3/4},\quad 
\kappa_t=\frac{\wtil{E}\wtil{M}-\wtil{F}\wtil{L}}{\wtil{E}\sqrt{\wtil{E}\wtil{G}-\wtil{F}^2}}
\end{aligned}
\end{align}
along the $u$-axis, where 
\begin{align}\label{fundamental}
\begin{aligned}
\wtil{E}&=|f_u|^2,& \wtil{F}&=\inner{f_u}{h},& \wtil{G}&=|h|^2,\\
\wtil{L}&=-\inner{f_u}{\nu_u},& \wtil{M}&=-\inner{h}{\nu_u},&\wtil{N}&=-\inner{h}{\nu_v}.
\end{aligned}
\end{align}
We consider the representation of $\kappa_i$ like as \eqref{inv_cusp}. 
We now introduce the following functions:
\begin{align}\label{christof1}
\wtil{\varGamma}{^{\,\,1}_{1\,1}}
&=\frac{\wtil{G}\wtil{E}_u-2\wtil{F}\wtil{F}_u+2\wtil{F}\inner{f_u}{h_u}}
{2(\wtil{E}\wtil{G}-\wtil{F}^2)}, & 
\wtil{\varGamma}{^{\,\,2}_{1\,1}}
&=\frac{2\wtil{E}\wtil{F}_u-2\wtil{E}\inner{f_u}{h_u}-\wtil{F}\wtil{E}_u}
{2(\wtil{E}\wtil{G}-\wtil{F}^2)}\notag \\
\wtil{\varGamma}{^{\,\,1}_{1\,2}}
&=\frac{2\wtil{G}\inner{f_u}{h_u}-\wtil{F}\wtil{G}_u}{2(\wtil{E}\wtil{G}-\wtil{F}^2)}, & 
\wtil{\varGamma}{^{\,\,2}_{1\,2}}
&=\frac{\wtil{E}\wtil{G}_u-2\wtil{F}\inner{f_u}{h_u}}{2(\wtil{E}\wtil{G}-\wtil{F}^2)}\\
\wtil{\varGamma}{^{\,\,1}_{2\,2}}
&=\frac{2\wtil{G}\wtil{F}_v-v\wtil{G}\wtil{G}_u-\wtil{F}\wtil{G}_v}{2(\wtil{E}\wtil{G}-\wtil{F}^2)}, & 
\wtil{\varGamma}{^{\,\,2}_{2\,2}}
&=\frac{\wtil{E}\wtil{G}_v-2\wtil{F}\wtil{F}_v+v\wtil{F}\wtil{G}_u}{2(\wtil{E}\wtil{G}-\wtil{F}^2)}\notag
\end{align}
where the functions $\wtil{E}, \wtil{F}$ and $\wtil{G}$ are defined in \eqref{fundamental} and 
$\wtil{E}_v=2v\inner{f_u}{h_u}$ holds. 
We call functions $\wtil{\varGamma}\genfrac{}{}{0pt}{2}{i}{j\,k}$ as in \eqref{christof1} {\it modified Christoffel symbols}.
\begin{lem}\label{seconddiff1}
Under the above notations, we have
\begin{align*}
f_{uu}&=\wtil{\varGamma}{^{\,\,1}_{1\,1}}f_u
+\wtil{\varGamma}{^{\,\,2}_{1\,1}}h+\wtil{L}\nu,\\
f_{uv}&=v\wtil{\varGamma}{^{\,\,1}_{1\,2}}f_u
+v\wtil{\varGamma}{^{\,\,2}_{1\,2}}h+v\wtil{M}\nu,\\
f_{vv}&=v\wtil{\varGamma}{^{\,\,1}_{2\,2}}f_u+
\left(1+v\wtil{\varGamma}{^{\,\,2}_{2\,2}}\right)h+v\wtil{N}\nu.
\end{align*}
\end{lem}
\begin{proof}
We now set the following: 
\begin{align*}
f_{uu}&=X_1 f_u+X_2 h+X_3\nu,\\
f_{uv}(=vh_u)&=Y_1 f_u+Y_2 h+Y_3\nu,\\
f_{vv}(=h+vh_v)&=Z_1 f_u+Z_2 h+Z_3\nu,
\end{align*}
where $X_i,Y_i,Z_i:U\to\R$ $(i=1,2,3)$ are $C^\infty$ functions. 

First we consider $f_{uu}$. 
By the definition of $\wtil{L}$ and $\inner{f_u}{\nu}=0$, we have $X_3=\wtil{L}$. 
Let us determine the functions $X_1$ and $X_2$. 
By direct calculations, we have
\begin{equation*}
\inner{f_{uu}}{f_u}=\wtil{E}X_1+\wtil{F}X_2,\quad 
\inner{f_{uu}}{h}=\wtil{F}X_1+\wtil{G}X_2.
\end{equation*}
Differentiating $\wtil{E}=\inner{f_u}{f_u}$ by $u$, we obtain $\inner{f_{uu}}{f_u}=\wtil{E}_u/2$. 
Moreover, since $\wtil{F}=\inner{f_u}{h}$, we have 
$\inner{f_{uu}}{h}=\wtil{F}_u-\inner{f_u}{h_u}$. 
Thus the above equations can be rewritten as 
\begin{equation*}
\begin{pmatrix} \dfrac{\wtil{E}_u}{2} \\ \wtil{F}_u-\inner{f_u}{h_u} \end{pmatrix}=
\begin{pmatrix} \wtil{E} & \wtil{F} \\ \wtil{F} & \wtil{G} \end{pmatrix} 
\begin{pmatrix} X_1 \\ X_2 \end{pmatrix}.
\end{equation*}
Since $\wtil{E}\wtil{G}-\wtil{F}^2>0$, one can solve this equation and get 
$X_i=\wtil{\varGamma}\genfrac{}{}{0pt}{2}{i}{1\,1}$ $(i=1,2)$. 

Next we consider $f_{uv}$. 
It follows that 
$\inner{f_{uv}}{\nu}=v\inner{h_u}{\nu}=-v\inner{h}{\nu_u}=v\wtil{M}=Y_3$ 
since $\inner{h}{\nu}=0$. 
For $Y_1$ and $Y_2$, by the similar computations as above, we get the following equation:
\begin{equation*}
v\begin{pmatrix} \inner{f_u}{h_u} \\ \dfrac{\wtil{G}_u}{2} \end{pmatrix}=
\begin{pmatrix} \wtil{E} & \wtil{F} \\ \wtil{F} & \wtil{G} \end{pmatrix} 
\begin{pmatrix} Y_1 \\ Y_2 \end{pmatrix}.
\end{equation*}
Therefore we have $Y_i=v\wtil{\varGamma}\genfrac{}{}{0pt}{2}{i}{1\,2}$ $(i=1,2)$. 

Finally we show the case of $f_{vv}$. 
$f_{vv}$ can be written as $f_{vv}=h+vh_v$ since $f_v=vh$. 
Thus the inner product of $f_{vv}$ and $\nu$ is calculated as 
$\inner{f_{vv}}{\nu}=v\inner{h_v}{\nu}=-v\inner{h}{\nu_v}=v\wtil{N}=Y_3$ 
since $\inner{h}{\nu}=0$ and $\inner{\nu}{\nu}=1$. 
For $Z_i$ $(i=1,2)$, we have the following equation
\begin{equation*}
\begin{pmatrix} \wtil{F}+v\left(\wtil{F}_v-\dfrac{v\wtil{G}_u}{2}\right) \\ 
\wtil{G}+\dfrac{v\wtil{G}_v}{2} \end{pmatrix}=
\begin{pmatrix} \wtil{E} & \wtil{F} \\ \wtil{F} & \wtil{G} \end{pmatrix} 
\begin{pmatrix} Z_1 \\ Z_2 \end{pmatrix}
\end{equation*}
by the similar calculations as above. 
Solving this equation, we have $Z_1=v\wtil{\varGamma}^{\,1}_{2\,2}$ 
and $Z_2=1+v\wtil{\varGamma}^{\,\,2}_{2\,2}$. %\qed
\end{proof}
Using Lemma \ref{seconddiff1}, we formulate the edge inflectional curvature 
$\kappa_i$ in our setting.
\begin{lem}\label{cuspidalcurv}
Under the above settings with $\eta\lambda(u,0)>0$, $\kappa_i$ can be expressed as 
\begin{equation}\label{ki}
\kappa_i=
\frac{(\wtil{E}\wtil{M}-\wtil{F}\wtil{L})(2\wtil{F}_u\wtil{E}-\wtil{E}\wtil{E}_{vv}-\wtil{E}_u\wtil{F})}
{2\wtil{E}^{5/2}(\wtil{E}\wtil{G}-\wtil{F}^2)}
+\frac{\wtil{E}\wtil{L}_u-\wtil{E}_u\wtil{L}}{\wtil{E}^{5/2}}
\end{equation}
along the $u$-axis. 
\end{lem}
\begin{proof}
We take an adapted coordinate system $(U;u,v)$ around a cuspidal edge $p$ 
with $\eta\lambda(u,0)=\det(f_u,h,\nu)(u,0)>0$. 
By Lemma \ref{seconddiff1}, $f_{uuu}$ is given by 
\begin{equation*}
f_{uuu}=\ast_1 f_u+\ast_2 h
+\left(\wtil{\varGamma}{^{\,\,1}_{1\,1}}\wtil{L}+\wtil{\varGamma}{^{\,\,2}_{1\,1}}\wtil{M}+\wtil{L}_u\right)\nu,
\end{equation*}
where $\ast_i$ $(i=1,2)$ are some functions. 
Thus it follows that 
\begin{equation*}
\det(f_u,f_{vv},f_{uuu})
=\left(\wtil{\varGamma}{^{\,\,1}_{1\,1}}\wtil{L}+\wtil{\varGamma}{^{\,\,2}_{1\,1}}\wtil{M}+\wtil{L}_u\right)
\det(f_u,h,\nu)
\end{equation*}
along the $u$-axis. 
Moreover, we have 
\begin{equation*}
\frac{\det(f_u,f_{vv},f_{uu})\inner{f_u}{f_{uu}}}{|f_u\times f_{vv}|}=\frac{\wtil{L}\wtil{E}_u}{2}.
\end{equation*}
Therefore we obtain 
\begin{align*}
\kappa_i(u)&=\frac{\det(f_u,f_{vv},f_{uuu})}{|f_u|^3|f_u\times f_{vv}|}(u,0)
-3\frac{\det(f_u,f_{vv},f_{uu})\inner{f_u}{f_{uu}}}{|f_u|^5|f_u\times f_{vv}|}(u,0)\\
&=\left(\frac{\wtil{\varGamma}{^{\,\,1}_{1\,1}}\wtil{L}+\wtil{\varGamma}{^{\,\,2}_{1\,1}}\wtil{M}+\wtil{L}_u}
{\wtil{E}^{3/2}}-\frac{3\wtil{L}\wtil{E}_u}{2\wtil{E}^{5/2}}\right)(u,0)\\
&=\left(
\frac{(\wtil{E}\wtil{M}-\wtil{F}\wtil{L})(2\wtil{F}_u\wtil{E}-\wtil{E}\wtil{E}_{vv}-\wtil{E}_u\wtil{F})}
{2\wtil{E}^{5/2}(\wtil{E}\wtil{G}-\wtil{F}^2)}
+\frac{\wtil{E}\wtil{L}_u-\wtil{E}_u\wtil{L}}{\wtil{E}^{5/2}}\right)(u,0),
\end{align*}
where we used the relation $\wtil{E}_{vv}=2\inner{f_u}{h_u}$ along the $u$-axis and the expressions of 
$\wtil{\varGamma}{^{\,\,i}_{1\,1}}$ $(i=1,2)$ as in \eqref{christof1}.%\qed
\end{proof}
By \eqref{inv_cusp} and \eqref{ki}, we see that 
\begin{equation}\label{eq:kikskt}
\kappa_i=\kappa_t\kappa_s+\dfrac{\kappa_\nu'}{\sqrt{\wtil{E}}}
\end{equation}
along the $u$-axis. 
In particular, if $\kappa_\nu'(p)=0$, then we have $\kappa_i(p)=\kappa_t(p)\kappa_s(p)$.

Next, we consider the non-degenerate singular point of the second kind. 
Let $f:\Sig\to\R^3$ be a front, $\nu$ its unit normal vector 
and $p$ a non-degenerate singular point of the second kind. 
Take an adapted coordinate system $(U;u,v)$ around $p$, 
and denote by $H$ the mean curvature of $f$ defined on $U\setminus\{v=0\}$. 
Then the {\it normalized cuspidal curvature} $\mu_c(p)$ is defined by $\mu_c(p)=2\hat{H}(p)$, 
where $\hat{H}=vH$ (cf. \cite{msuy}). 
It is known that $f$ is a front at $p$ if and only if $\mu_c(p)\neq0$ holds (\cite[Proposition 4.2]{msuy}).

\subsection{Principal curvatures of fronts}
We recall behavior of principal curvature of fronts near non-degenerate singular points.
Let $f:\Sig\to\R^3$ be a front, $\nu$ a unit normal vector of $f$ 
and $p\in\Sig$ a non-degenerate singular point of $f$. 
Then we take an adapted coordinate system $(U;u,v)$ around $p$. 
Under this coordinate system, there exists a $C^\infty$ map 
$h:U\to\R^3\setminus\{0\}$ such that $df(\eta)=vh$ and $\langle{h,\nu}\rangle=0$. 
If $p$ is a cuspidal edge, then $df(\eta)=f_v=vh$ and the pair $\{f_u,h,\nu\}$ gives a frame. 
On the other hand, if $p$ is of the second kind, 
$f_u=vh-\epsilon(u) f_v$ and $\{h,f_v,\nu\}$ gives a frame 
since $df(\eta)=f_u+\epsilon(u)f_v=vh$ holds. 

Let $K$ and $H$ be the Gaussian and the mean curvature of $f$ defined on $U\setminus\{v=0\}$, 
namely, the set of regular points of $f$ on $U$. 
Then these functions are bounded $C^\infty$ functions. 
Using these functions, we define two functions $\kappa_j$ $(j=1,2)$ as 
$$\kappa_1=H+\sqrt{H^2-K},\quad \kappa_2=H-\sqrt{H^2-K}$$
on $U\setminus\{v=0\}$. 
These functions are dealt with {\it principal curvatures} of $f$. 
It is known that one of $\kappa_j$ (\/$j=1,2$\/) can be extended as a bounded $C^\infty$ function on $U$ 
(\cite[Theorem 3.1]{t2}, see also \cite{mu}). 
Moreover, one can take a {\it principal vector $\bm{V}$ 
with respect to bounded principal curvature} (\cite[(3.2) and (3.3)]{t2}).

We write $\kappa$ as the bounded principal curvature of $f$ on $U$ 
and $\tilde{\kappa}$ as the unbounded one. 
We remark that $\kappa(p)=\kappa_\nu(p)$ holds. 
On the other hand, for a unbounded principal curvature $\tilde{\kappa}$, we have the following.
\begin{prop}\label{prop:change}
Under the above setting, the unbounded principal curvature $\tilde{\kappa}$ 
changes the sign across the singular curve.
\end{prop} 
\begin{proof}
We set $\hat{\kappa}=\lambda\tilde{\kappa}$, where $\lambda=\det(f_u,f_v,\nu)$. 
Then by \cite[Remark 3.2]{t2}, if $p$ is a cuspidal edge of $f$, 
$\hat{\kappa}$ is proportional to a nonzero functional multiple of 
the {cuspidal curvature} $\kappa_c$ along the singular curve. 
On the other hand, if $p$ is of the second kind, $\hat{\kappa}(p)$ is proportional to a nonzero real multiple of 
the {normalized cuspidal curvature} $\mu_c(p)(\neq0)$. 
Thus $\hat{\kappa}(p)\neq0$ holds, 
and hence we may assume that $\hat{\kappa}>0$ near $p$. 
The function $\lambda$ changes the sign across the singular curve. 
Therefore we have the assertion.
\end{proof}
\begin{dfn}
Under the above settings, 
a point $p$ is a {\it ridge point} if $\bm{V}\kappa(p)=0$ holds, 
where $\bm{V}\kappa$ means directional derivative of $\kappa$ in the direction $\bm{V}$. 
Moreover, $p$ is a {\it $k$-th order ridge point} if $\bm{V}^{(m)}\kappa(p)=0$ $(1\leq m\leq k)$ 
and $\bm{V}^{(k+1)}\kappa(p)\neq0$ holds, 
where $\bm{V}^{(m)}\kappa$ means $m$-th order directional derivative of $\kappa$ in the direction $\bm{V}$.
\end{dfn}
Porteous introduced the concepts of ridge points on regular surfaces. 
He showed that a ridge point on a surface corresponds to an $A_3$ singularity of the distance squared functions on it.
For more other characterizations of ridge points on surfaces, 
see \cite{bgm,bgt,bt,fh,po1,po}.

\section{Gauss maps of fronts}
\subsection{Singularities of maps from the plane into the plane}
We recall singularities of a map $f:\R^2\to\R^2$ to consider singularities of Gauss maps of fronts. 
Whitney showed that generic singularities of these mappings are a {\it fold} and a {\it cusp}, 
which are $\mathcal{A}$-equivalent to the germs $(u,v)\mapsto(u,v^2)$ and $(u,v)\mapsto(u,v^3+uv)$ 
at the origin, respectively (see Figure \ref{fig:sing1}). 
\begin{figure}[htbp]
  \begin{center}
    \begin{tabular}{c}

      % 1
      \begin{minipage}{0.33\hsize}
        \begin{center}
          \includegraphics[clip, width=3cm]{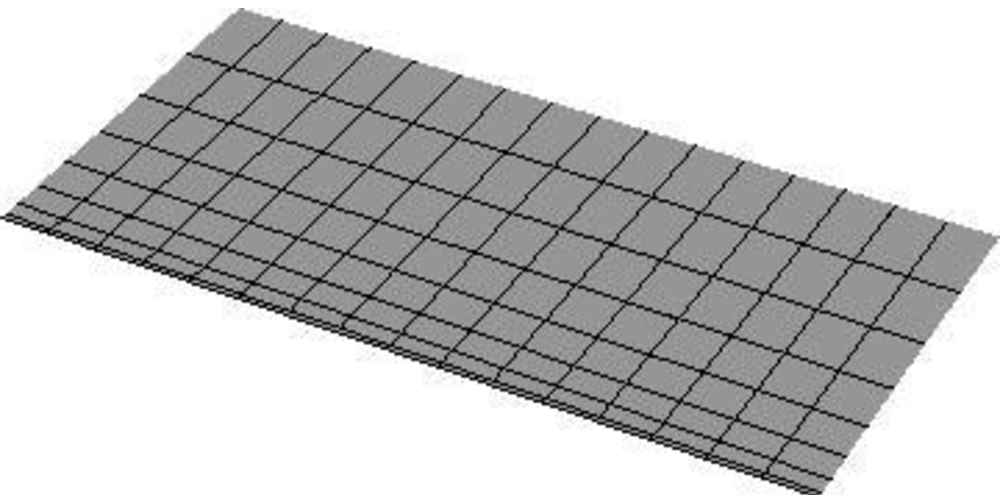}
          \hspace{1.6cm} fold
        \end{center}
      \end{minipage}

      % 2
      \begin{minipage}{0.33\hsize}
        \begin{center}
          \includegraphics[clip, width=3cm]{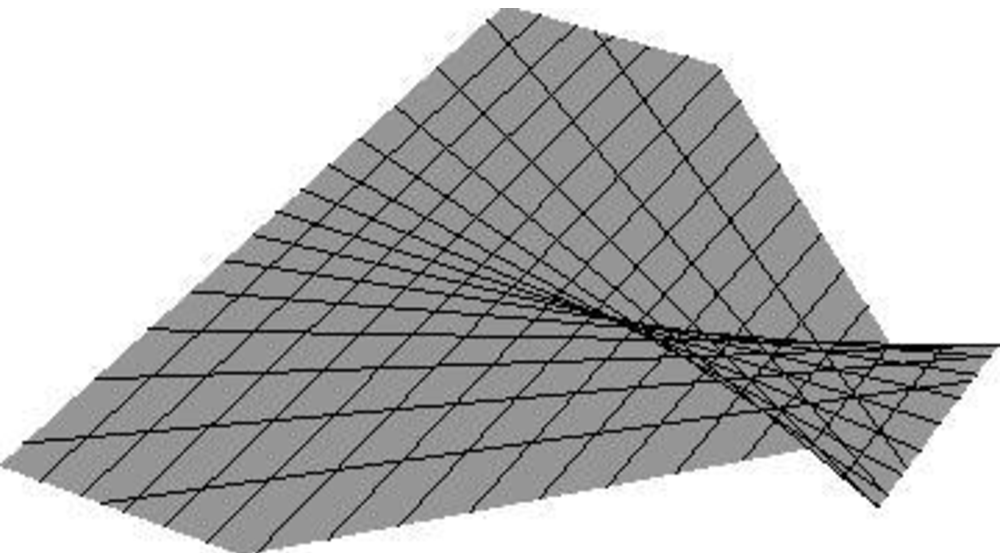}
          \hspace{1.6cm} cusp
        \end{center}
      \end{minipage}

    \end{tabular}
    \caption{Fold and cusp.}
    \label{fig:sing1}
  \end{center}
\end{figure}
Moreover, Rieger \cite{r} classified singularities of $C^\infty$ map germs 
from a plane into a plane with corank $1$ and 
$\mathcal{A}_e$-codimension $\leq6$. 
Singularities called {\it lips}, {\it beaks} and {\it swallowtail} are the map germs $\mathcal{A}$-equivalent to 
$(u,v)\mapsto(u,v^3+u^2v)$, $(u,v^3-u^2v)$ and $(u,v^4+ uv)$ at the origin, respectively (see Figure \ref{fig:sing2}). 
These singularities are $\mathcal{A}_e$-codimension $1$. 
\begin{figure}[htbp]
  \begin{center}
    \begin{tabular}{c}

      % 1
      \begin{minipage}{0.3\hsize}
        \begin{center}
          \includegraphics[clip, width=3cm]{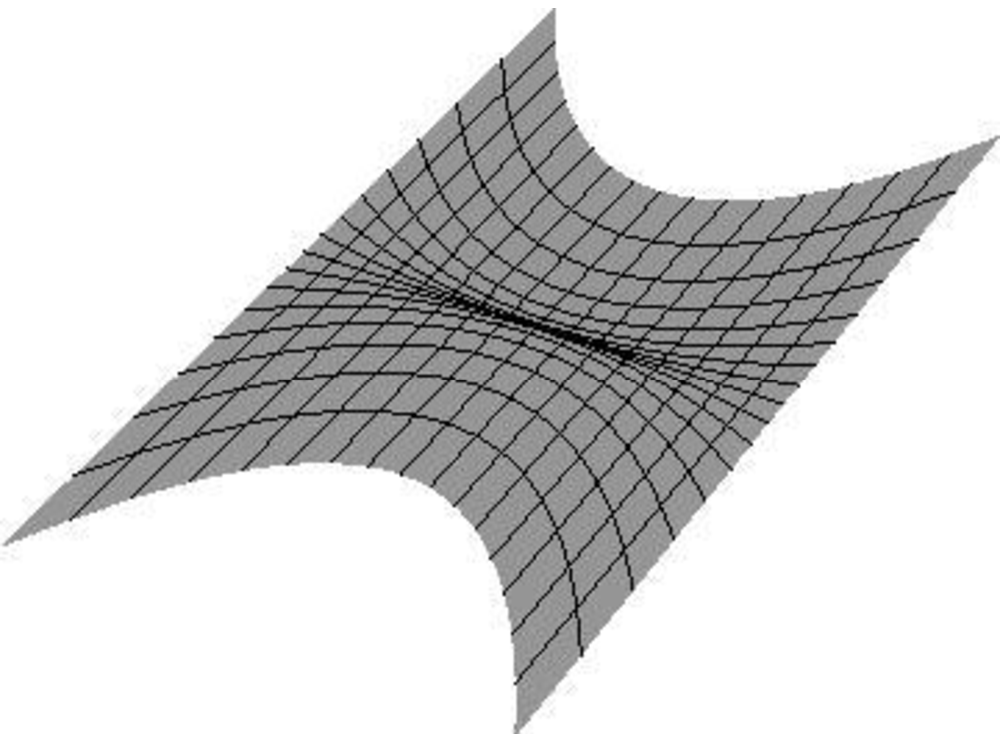}
          \hspace{1.6cm} lips
        \end{center}
      \end{minipage}

      % 2
      \begin{minipage}{0.3\hsize}
        \begin{center}
          \includegraphics[clip, width=3cm]{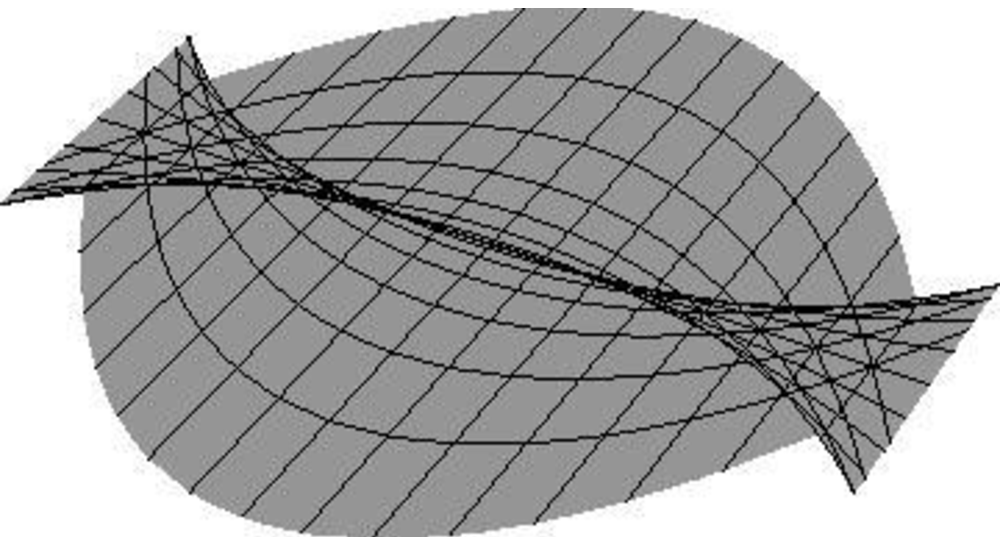}
          \hspace{1.6cm} beaks
        \end{center}
      \end{minipage}

      % 3
      \begin{minipage}{0.3\hsize}
        \begin{center}
          \includegraphics[clip, width=3cm]{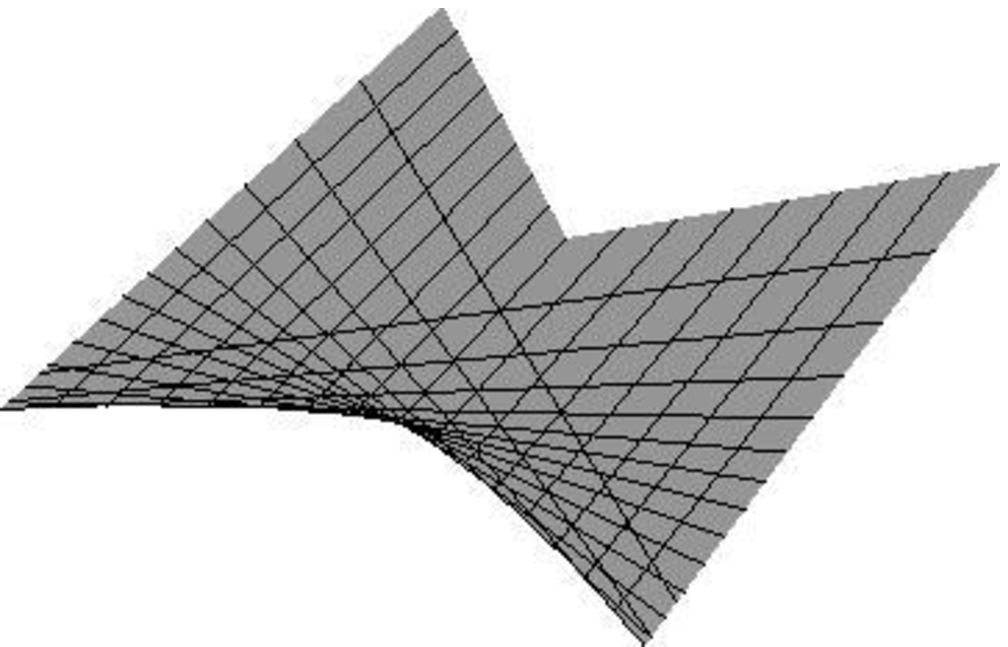}
          \hspace{1.6cm} swallowtail
        \end{center}
      \end{minipage}

    \end{tabular}
    \caption{Lips, beaks and swallowtail.}
    \label{fig:sing2}
  \end{center}
\end{figure}

Let $f:\R^2\to\R^2$ be a $C^\infty$ map. 
Then we set a function $\Lambda:\R^2\to\R$ by 
$$\Lambda(u,v)=\det(f_u,f_v)(u,v)$$
for some local coordinates $(u,v)$. 
We call $\Lambda$ the {\it identifier of singularity} of $f$. 
By the definition of $\Lambda$, we see that the set of singular points $S(f)$ of $f$ 
coincides with $\Lambda^{-1}(0)$.  
Take a corank $1$ singular point $p$ of $f$. 
Then there exist a neighborhood $V$ of $p$ and a never vanishing vector field $\eta\in\mathfrak{X}(V)$ 
such that $df_q(\eta_q)=0$ holds for any $q\in S(f)\cap V$. 
We call $\eta$ the {\it null vector field}. 
A singular point $p$ is called a {\it non-degenerate singular point} if $d\Lambda(p)\neq0$. 

\begin{fact}[\cite{s1,w}]\label{singular_plane}
Let $f:\R^2\to\R^2$ be a $C^\infty$ map and $p\in\R^2$ a singular point of $f$. 
Then 
\begin{enumerate}
\item $f$ at $p$ is a fold if and only if $\eta\Lambda(p)\neq0$.
\item $f$ at $p$ is a cusp if and only if $d\Lambda(p)\neq0$, $\eta\Lambda(p)=0$ 
and $\eta\eta\Lambda(p)\neq0$.
\item $f$ at $p$ is a swallowtail if and only if $d\Lambda(p)\neq0$, $\eta\Lambda(p)=\eta\eta\Lambda(p)=0$ 
and $\eta\eta\eta\Lambda(p)\neq0$.
\item $f$ at $p$ is a lips if and only if $d\Lambda(p)=0$ and $\det\hess(\Lambda(p))>0$.
\item $f$ at $p$ is a beaks if and only if $d\Lambda(p)=0$, $\det\hess(\Lambda(p))<0$ 
and $\eta\eta\Lambda(p)\neq0$.
\end{enumerate}
\end{fact}
Criteria for more degenerate corank $1$ singularities of maps from the plane into the plane 
are given by Kabata \cite{kaba}. 

\subsection{Singularities of Gauss maps of fronts}
Let $f:\Sig\to\R^3$ be a front, $\nu:\Sig\to S^2$ the Gauss map of $f$ 
and $p\in\Sig$ a non-degenerate singular point. 
It is obviously that $\nu$ is a $C^\infty$ map between $2$-dimensional manifolds.

We consider the singularities of Gauss map of a front. 
\begin{prop}\label{singptgauss}
Let $f:\Sig\to\R^3$ be a front, $\nu$ the Gauss map and 
$p\in \Sig$ a non-degenerate singular point of $f$. 
Then the Gauss map $\nu$ is also singular at $p$ if and only if $\kappa(p)=0$, 
where $\kappa$ is a bounded principal curvature near $p$. 
Moreover, $p$ is a non-degenerate singular point of $\nu$ if and only if $p$ is not a singular point of $\kappa$.
\end{prop}
\begin{proof}
Since $p$ is a non-degenerate singular point of a front $f$, 
there exist a neighborhood $V\subset\Sig$ of $p$ and a regular curve $\gamma:(-\epsilon,\epsilon)\to V$ 
such that $\im(\gamma)=S(f)\cap V$. 
Let $(u,v)$ be a coordinate system of $V$. 
Then an identifier of singularity $\Lambda:V\to\R$ of $\nu$ is given by 
\begin{equation}\label{Lambda}
\Lambda(u,v)=\det(\nu_u,\nu_v,\nu)(u,v).
\end{equation}
By the Weingarten formula, 
the function $\Lambda$ as in \eqref{Lambda} can be expressed as 
$$\Lambda=K\lambda=\kappa\hat{\kappa},$$
where $\hat{\kappa}=\lambda\tilde{\kappa}$. 
Since $\hat{\kappa}(p)\neq0$, $p$ is a singular point of $\nu$ if and only if $\kappa(p)=0$ holds. 

By this argument, we can regard $\til{\Lambda}=\kappa$ as the identifier of singularity of $\nu$. 
Hence $p$ is a non-degenerate singular point of the Gauss map $\nu$ if and only if 
$\kappa(p)=0$ and $(\partial_u\kappa(p),\partial_v\kappa(p))\neq(0,0)$, 
where $\partial_u=\partial/\partial u$ and $\partial_v=\partial/\partial v$. 
\end{proof}

It is known that a non-degenerate singular point of a front is also a singular point of its Gauss map 
if and only if the limiting normal curvature $\kappa_\nu$ vanishes at $p$ (cf. \cite{msuy}). 
Moreover, in such a case, the Gaussian curvature of a front is {\it rationally bounded} 
(see \cite[Theorem B and Corollary C]{msuy}). 
For detailed definition, see \cite[Definition 3.4]{msuy}. 

By Proposition \ref{singptgauss}, we may assume that the set of singular points $S(\nu)$ of the Gauss map $\nu$ 
is given as $S(\nu)=\{q\in U\ |\ \kappa(q)=0\}$, and $\nu$ has only corank $1$ singularities at $p$. 
As in the case of regular surfaces, 
we call a point $q\in S(\nu)=\kappa^{-1}(0)$ and a curve given by $\kappa^{-1}(0)$ 
a {\it parabolic point} and  a {\it parabolic curve} of $f$, respectively.
\begin{thm}\label{typegauss}
Let $f:\Sig\to\R^3$ be a front, $\nu$ the Gauss map, 
$p\in\Sig$ a non-degenerate singular point of $f$ 
and $\kappa$ a bounded principal curvature at $p$. 
Assume that $p$ is a parabolic point of $f$. 
Then the following assertions hold.
\begin{enumerate}
\item\label{critgauss:1} Suppose that $p$ is a regular point of $\kappa$. 
\begin{itemize}
\item $p$ is a fold of $\nu$ if and only if $p$ is not a ridge point.
\item $p$ is a cusp of $\nu$ if and only if $p$ is a first order ridge point. 
\item $p$ is a swallowtail of $\nu$ if and only if $p$ is a second order ridge point.
\end{itemize}
\item\label{critgauss:2} Suppose that $p$ is a singular point of $\kappa$.
\begin{itemize}
\item $p$ is a lips of $\nu$ if and only if $\det\hess(\kappa(p))>0$.
\item $p$ is a beaks of $\nu$ if and only if $\det\hess(\kappa(p))<0$  
and $p$ is a first order ridge point.
\end{itemize}
\end{enumerate}
\end{thm}
\begin{proof}
Assume that $p$ is a non-degenerate singular point of the second kind, 
and take an adapted coordinate system $(U;u,v)$ centered at $p$. 
Then there exists a map $h:U\to\R^3\setminus\{0\}$ such that $f_u=vh-\epsilon(u)f_v$. 
We set functions as follows:
\begin{equation}\label{fundamentals}
\what{E}=|h|^2,\ \what{F}=\langle{h,f_v}\rangle,\ \what{G}=|f_v|^2,\ 
\what{L}=-\langle{h,\nu_u}\rangle,\ \what{M}=-\langle{h,\nu_v}\rangle,\ \what{N}=-\langle{f_v,\nu_v}\rangle,
\end{equation}
where $|\bm{x}|^2=\langle{\bm{x},\bm{x}}\rangle$ for any $\bm{x}\in\R^3$.
Note that $\what{E}\what{G}-\what{F}^2\neq0$ and $\what{L}(p)\neq0$ hold. 
We assume that $\what{L}(p)>0$. 
In this case, principal curvatures $\kappa$ and $\tilde{\kappa}$ can be written as 
\begin{equation}\label{principal}
\kappa=\frac{2((\what{L}+\epsilon(u)\what{M})\what{N}-v\what{M}^2)}{\what{A}+\what{B}},\quad
\tilde{\kappa}=\frac{2((\what{L}+\epsilon(u)\what{M})\what{N}-v\what{M}^2)}{\what{A}-\what{B}},
\end{equation}
where 
\begin{align*}
\what{A}&=\what{G}(\what{L}+\epsilon(u)\what{M})-2v\what{F}\what{M}+v\what{E}\what{N},\\%\label{sec_A}\\
\what{B}&=\sqrt{\what{A}^2-4v(\what{E}\what{G}-\what{F}^2)\left((\what{L}+\epsilon(u)\what{M})\what{N}
-v\what{M}^2\right)}
\end{align*}
(cf. \cite[Theorem 3.1]{t2}). 
Moreover, by using functions as in \eqref{fundamentals}, $\nu_u$ and $\nu_v$ are expressed as 
\begin{align*}
\nu_u&=\frac{\what{F}(v\what{M}-\epsilon\what{N})-\what{G}\what{L}}{\what{E}\what{G}-\what{F}^2}h+
\frac{\what{F}\what{L}-\what{E}(v\what{M}-\epsilon\what{N})}{\what{E}\what{G}-\what{F}^2}f_v,\\
\nu_v&=\frac{\what{F}\what{N}-\what{G}\what{M}}{\what{E}\what{G}-\what{F}^2}h+
\frac{\what{F}\what{M}-\what{E}\what{N}}{\what{E}\what{G}-\what{F}^2}f_v
\end{align*}
by \cite[Lemma 2.8]{t2}.

First, we prove $(1)$. 
In this case, $S(\nu)$ can be parametrized by a regular curve and there exists a null vector field 
$\bm{X}=X_1\partial_u+X_2\partial_v$, where $X_i$ ($i=1,2$) are $C^\infty$ functions on $U$. 
We find explicit form of $\bm{X}$. 
If $X$ is a null vector field, then $d\nu(\bm{X})=0$ holds on $S(\nu)$. 
By the above expressions, this is equivalent to the equation 
\begin{equation*}
\begin{pmatrix}\what{E} & \what{F}\\ \what{F} & \what{G}\end{pmatrix}^{-1}
\begin{pmatrix} \what{L} & \what{M}\\ v\what{M}-\epsilon(u)\what{N} & \what{N} \end{pmatrix}
\begin{pmatrix} X_1\\ X_2\end{pmatrix}=
\begin{pmatrix}0\\0\end{pmatrix}.
\end{equation*}
Since $\what{L}$ does not vanish at $p$, 
we can take $X_1=-\what{M}$ and $X_2=\what{L}$. 
Moreover, $\bm{X}$ can be extended on $U$ by the form 
$$\bm{X}=(-\what{M}+\kappa\what{F})\partial_u+(\what{L}-\kappa(v\what{E}-\epsilon(u)\what{F}))\partial_v(=\bm{V})$$
since $\kappa=0$ on $S(\nu)$ (cf. \cite[(3.2)]{t2}). 
This implies that the principal vector $\bm{V}$ can be regarded as $\bm{X}$. 
Moreover, the discriminant function of $\nu$ is given by $\til{\Lambda}=\kappa$. 
By Fact \ref{singular_plane} and the definition of a (higher order) ridge point, 
assertion $(1)$ holds. 

Next we prove $(2)$. 
By the above arguments, a null vector satisfies $\bm{X}=\bm{V}$. 
Now the equation $\det\hess(\til{\Lambda}(p))=\det\hess(\kappa(p))$ holds.  
By Fact \ref{singular_plane} and the definition of the first order ridge point, the result follows. 

For the case of cuspidal edges, we can show assertions in the similar way. 
\end{proof}
%A point $p$ at which $\nu$ has a cusp singularity is called a {\it godron point} (see \cite{suy2}). 

We remark that this kind of characterizations of fold and cusp singularities of Gauss maps for regular surfaces 
are given (cf. \cite{bgm,ifrt}), 
and stability of Gauss maps of regular surfaces is studied in \cite{bl}. 
We also remark that relationship between $A_{k+1}$-inflection points on immersed hypersurfaces 
and $A_k$-Morin singularities on corresponding affine Gauss map are known (see \cite{suy2}).

\subsection{Contact between parabolic curves and singular curves of cuspidal edges}
We consider contact between a parabolic curve and a singular curve of a cuspidal edge at $p$. 
\begin{dfn}\label{contact}
Let $\alpha:I\ni t\mapsto\alpha(t)\in\R^2$ be a regular plane curve. 
Let $\beta$ be an another plane curve defined by the zero set of a smooth function $F:\R^2\to\R$. 
We say that {\it $\alpha$ has $(k+1)$-point contact at $t_0\in I$ with $\beta$} 
if the composite function $g(t)=F(\alpha(t))$ satisfies 
$$g(t_0)=g'(t_0)=\cdots=g^{(k)}(t_0)=0,\quad g^{(k+1)}(t_0)\neq0,$$
where $g^{(i)}=d^ig/dt^i$ $(1\leq i\leq k+1)$ (see Figure \ref{fig:contact}). 
Moreover, $\alpha$ has {\it at least $(k+1)$-point contact at $t_0$ with $\beta$} if 
the function $g(t)=F(\alpha(t))$ satisfies 
$$g(t_0)=g'(t_0)=\cdots=g^{(k)}(t_0)=0.$$ 
In this case, we call the integer $k$ the {\it order} of contact. 
\end{dfn} 

\begin{figure}[htbp]
  \begin{center}
    \begin{tabular}{c}

      % 1
      \begin{minipage}{0.3\hsize}
        \begin{center}
          \includegraphics[clip, width=3cm]{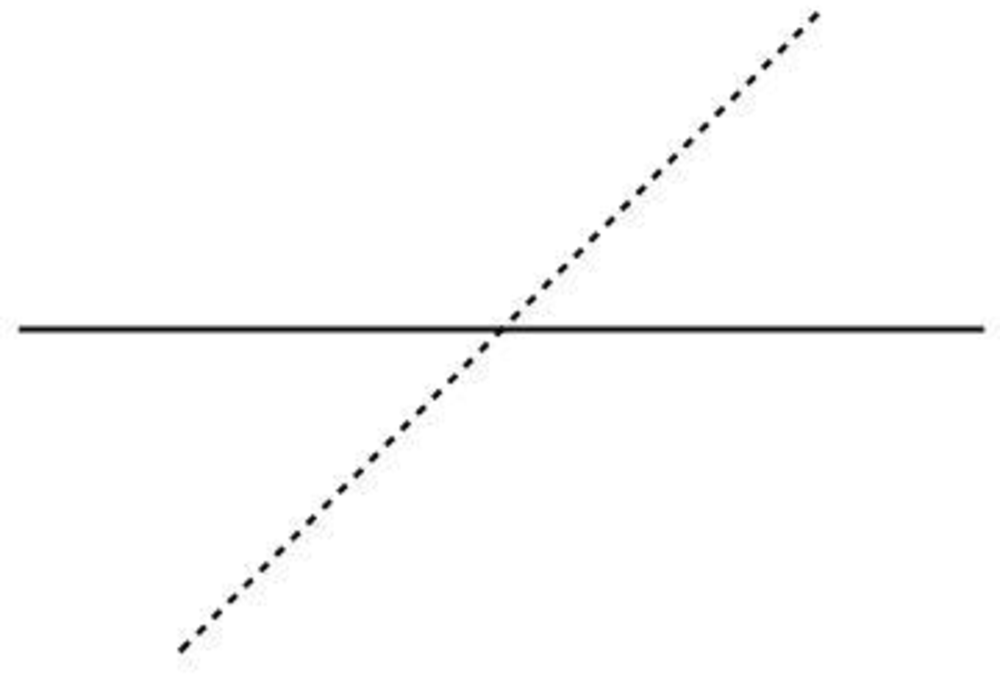}
          \hspace{1.6cm} $1$-point contact
        \end{center}
      \end{minipage}

      % 2
      \begin{minipage}{0.3\hsize}
        \begin{center}
          \includegraphics[clip, width=3cm]{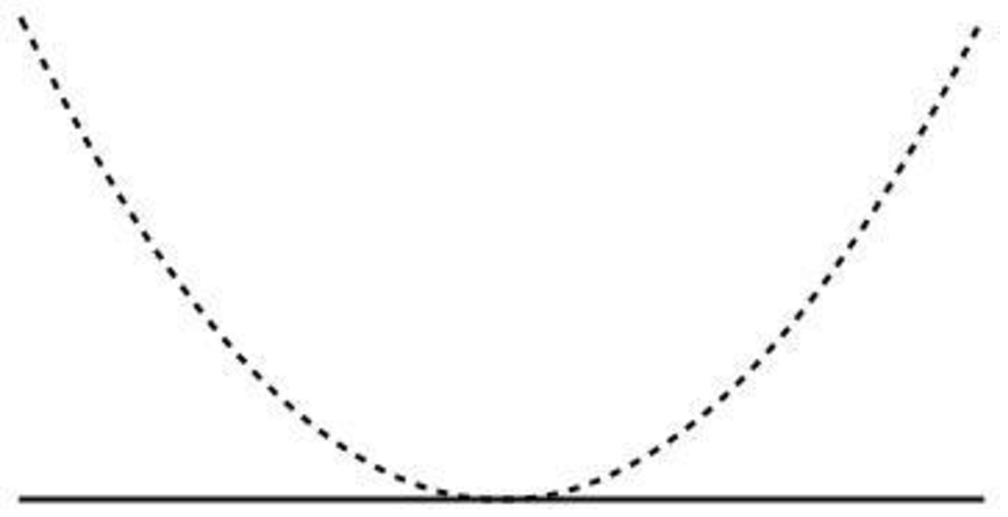}
          \hspace{1.6cm} $2$-point contact
        \end{center}
      \end{minipage}

      % 3
      \begin{minipage}{0.3\hsize}
        \begin{center}
          \includegraphics[clip, width=3cm]{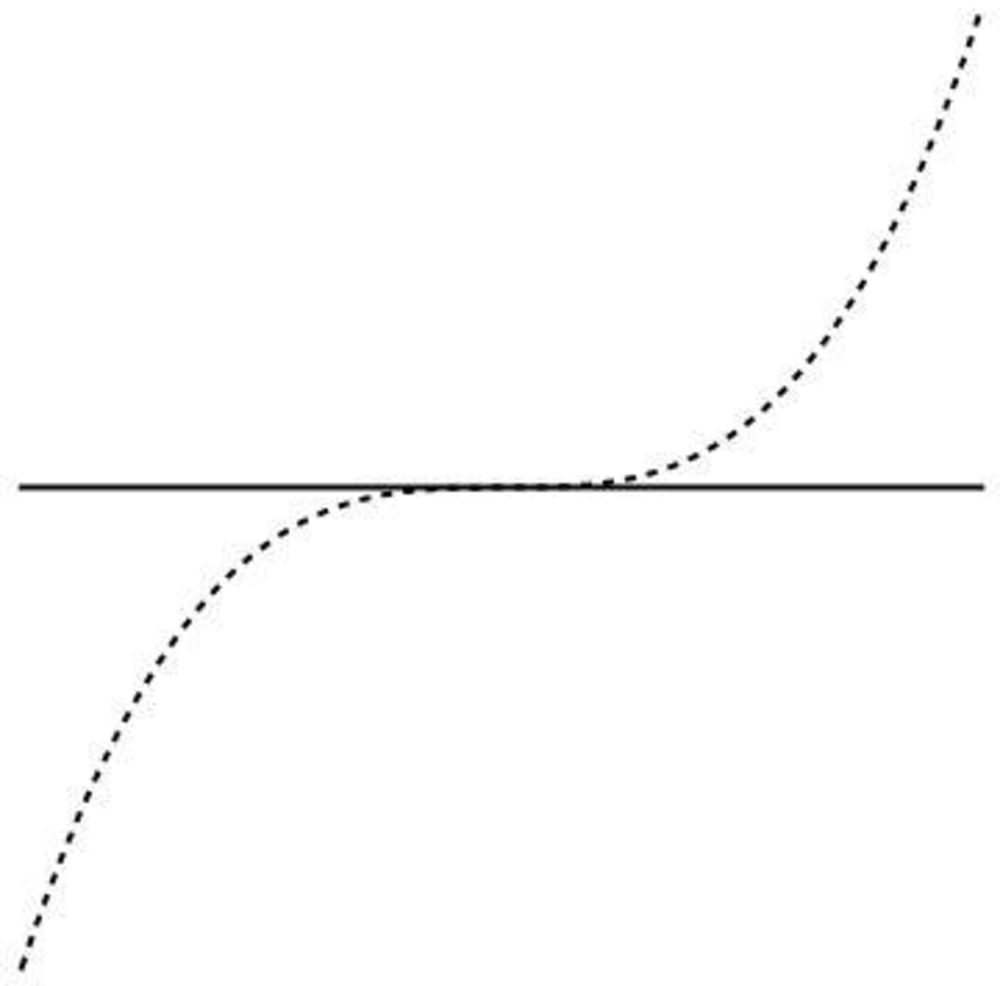}
          \hspace{1.6cm} $3$-point contact
        \end{center}
      \end{minipage}

    \end{tabular}
    \caption{Contact between $(t,0)$ and $F(x,y)=x^k-y$ at the origin, 
    where $k=1$ (left), $k=2$ (center) and $k=3$ (right).}
    \label{fig:contact}
  \end{center}
\end{figure}

It is known that the curve $\alpha$ has $(k+1)$-point contact at $t_0$ with $\beta$ 
if and only if the composite function $g$ has an $A_{k}$ singularity at $t_0$, 
where a function $f:\R\to\R$ has an {\it$A_k$ singularity} at $t_0\in \R$ if 
$f^{(i)}(t_0)=0$ $(1\leq i\leq k)$ and $f^{(k+1)}(t_0)\neq0$ hold (cf. \cite{bg,ifrt}).

First, we show the following lemma.
\begin{lem}\label{regularparabolic}
Let $f:\Sig\to\R^3$ be a front, $\nu$ its Gauss map and $p\in\Sig$ a cuspidal edge of $f$. 
Let $\kappa$ be a bounded principal curvature near $p$. 
Then the parabolic curve defined by $\kappa=0$ is regular at $p$ if and only if 
either $\kappa_\nu'\neq0$ or $4\kappa_t^2+\kappa_s\kappa_c^2\neq0$ holds at $p$.
\end{lem}
\begin{proof}%[Proof of Lemma \ref{regularparabolic}]
Let $(U;u,v)$ be an adapted coordinate system around $p$ satisfying $\eta\lambda>0$ on the $u$-axis. 
Then the parabolic curve defined by $\kappa=0$ is regular at $p$ 
if and only if $(\partial_u\kappa,\partial_v\kappa)\neq(0,0)$ at $p$. 
We now remark that $\kappa(u,0)=\kappa_\nu(u)$. 
Thus $\partial_u\kappa(p)\neq0$ if and only if $\kappa_\nu'(p)\neq0$ holds. 
On the other hand, we obtain 
\begin{equation}\label{eq:kv}
\partial_v\kappa
=-\frac{1}{2\kappa_c}(4\kappa_t^2+\kappa_s\kappa_c^2)
\left(\frac{\wtil{E}\wtil{G}-\wtil{F}^2}{\wtil{E}}\right)^{1/4} 
\end{equation}
along the $u$-axis by \cite[Proposition 2.8]{t3}. 
Thus we have the assertion.
\end{proof}
We note that the condition $4\kappa_t(p)^2+\kappa_s(p)\kappa_c(p)^2\neq0$ 
implies that $p$ is not a {\it sub-parabolic point} of $f$ (see \cite{t3}). 
\begin{prop}\label{kappanu}
Let $f:\Sig\to\R^3$ be a front, $\nu$ the Gauss map of $f$ and $p\in\Sig$ a cuspidal edge of $f$. 
Let $\kappa$ be the bounded principal curvature near $p$. 
Suppose that $\kappa(p)=0$ and $(\partial_u\kappa(p),\partial_v\kappa(p))\neq(0,0)$. 
Then the singular curve $\gamma$ passing through $p$ has $(k+1)$-point contact $(k\geq1)$ with 
the parabolic curve defined by $\kappa^{-1}(0)$ at $p$ 
if and only if $4\kappa_t(p)^2+\kappa_s(p)\kappa_c(p)^2\neq0$ and 
$$\kappa_\nu(p)=\kappa_\nu'(p)=\cdots=\kappa_\nu^{(k)}(p)=0\ \textit{and}\ \kappa_\nu^{(k+1)}(p)\neq0$$
hold.
\end{prop}
\begin{proof}
Take an adapted coordinate system $(U;u,v)$ centered at $p$. 
Then the singular curve $\gamma$ is given as $\gamma(u)=(u,0)$. 
Moreover, it follows that $\kappa(u,0)=\kappa_\nu(u)$ holds (\cite[Theorem 3.1]{t2}). 
When $\partial_u\kappa(p)=\kappa_\nu'(p)=0$, the parabolic curve $\kappa^{-1}(0)$ is regular at $p$ 
if and only if $\partial_v\kappa(p)\neq0$. 
This is equivalent to $4\kappa_t(p)^2+\kappa_s(p)\kappa_c(p)^2\neq0$ by Lemma \ref{regularparabolic}. 
Thus we have the assertion by the definition of contact of two curves (see  Definition \ref{contact}).
\end{proof}
By this proposition, we have the following.
\begin{cor}\label{boundedness}
Let $f:\Sig\to\R^3$ be a front and $p\in\Sig$ a cuspidal edge of $f$. 
Then the Gaussian curvature $K$ of $f$ is rationally bounded at $p$
if and only if a parabolic curve passes through $p$. 
Moreover, $K$ is rationally continuous at $p$ if 
the singular curve passing through $p$ has at least $2$-point contact with a parabolic curve at $p$. 
\end{cor}
\begin{proof}
It is known that the Gaussian curvature $K$ of $f$ is rationally bounded (resp. rationally continuous) 
if and only if $\kappa_\nu(p)=0$ (resp. $\kappa_\nu(p)=\kappa_\nu'(p)=0$) \cite[Corollary 3.12]{msuy}. 
Thus the assertion follows from Proposition \ref{kappanu}.
\end{proof}
This implies that behavior of a bounded principal curvature of a cuspidal edge 
is closely related to rational boundedness of the Gaussian curvature.

Conversely, the following assertion holds.
\begin{prop}\label{godroncusp}
Let $f:\Sig\to\R^3$ be a front, $\nu:\Sig\to S^2$ the Gauss map of $f$ and $p$ a cuspidal edge of $f$. 
Suppose that the Gaussian curvature $K$ of $f$ is rationally continuous at $p$. 
Then  $p$ is a cusp of $\nu$ %, namely, a godron point of $f$ 
if and only if $\kappa_t(p)=0$ and $\kappa_s(p)\kappa_i'(p)\neq0$.
\end{prop} 
\begin{proof}%[Proof of Proposition \ref{godroncusp}]
Let $(U;u,v)$ be an adapted coordinate system around $p$ with $\eta\lambda(u,0)>0$. 
We denote by $\kappa$ and $\bm{V}=V_1\partial_u+V_2\partial_v$ 
a bounded $C^\infty$ principal curvature of $f$ on $U$ 
and a principal vector relative to $\kappa$, respectively. 
We note that $V_1(p)\neq0$.
Since $K$ is rationally continuous at $p$, $\kappa_\nu=\kappa_\nu'=0$ at $p$. 
This implies that $\kappa(p)=\partial_u\kappa(p)=0$. 
In this case, the parabolic curve $\kappa^{-1}(0)$ is regular at $p$ 
if and only if $\partial_v\kappa(p)\neq0$, that is, 
$4\kappa_t^2+\kappa_s\kappa_c^2\neq0$ at $p$ by Lemma \ref{regularparabolic}. 
Thus it follows that 
$$\bm{V}\kappa=V_1\partial_u\kappa+V_2\partial_v\kappa=V_2\partial_v\kappa$$
at $p$. 
Since $V_2=-\kappa_t\sqrt{\wtil{E}\wtil{G}-\wtil{F}^2}$ holds along the $u$-axis (\cite[Proposition 3.3]{t2}), 
$\bm{V}\kappa(p)=0$ if and only if $\kappa_t(p)=0$. 

We calculate second order directional derivative $\bm{V}^{(2)}\kappa(p)$ under the assumptions that 
$\kappa_\nu=\kappa_\nu'=\kappa_t=0$ hold at $p$. 
By a direct computation, we have 
\begin{equation*}
\bm{V}^{(2)}\kappa=V_1\partial_u(\bm{V}\kappa)+V_2\partial_v(\bm{V}\kappa).
\end{equation*}
Since $V_2=0$ at $p$, 
$$\bm{V}^{(2)}\kappa(p)=V_1(p)\partial_u(\bm{V}\kappa)(p)$$ 
holds. 
By \eqref{eq:kikskt} and \eqref{eq:kv}, $\bm{V}\kappa$ can be written as 
\begin{equation*}
\bm{V}\kappa=\frac{1}{2\kappa_c}(4\kappa_t^3+\kappa_i\kappa_c^2)
\left(\frac{(\wtil{E}\wtil{G}-\wtil{F}^2)^3}{\wtil{E}}\right)^{1/4}
\end{equation*}
along the $u$-axis.
Thus $\partial_u(\bm{V}\kappa)\neq0$ at $p$ 
if and only if $12\kappa_t^2 \kappa_t'+\kappa_i' \kappa_c^2+2\kappa_i \kappa_c \kappa_c'\neq0$ at $p$ 
by the assumptions above. 
When $\kappa_\nu'(p)=0$, $\kappa_i(p)=\kappa_s(p)\kappa_t(p)$ holds (\eqref{eq:kikskt}, see also \cite[Theorem 4.4]{ms}). 
Hence $\partial_u(\bm{V}\kappa)\neq0$ at $p$ if and only if $\kappa_i'(p)\neq0$. 
Summing up, we have the assertion.
\end{proof}

For a cuspidal edge, it is known that 
{\it the singular locus $\hat{\gamma}=f\circ\gamma$ is a line of curvature of $f$ 
if and only if the cusp-directional torsion $\kappa_t$ 
identically vanishes along the singular curve $\gamma$} 
(\cite[Proposition 3.3]{t2}, see also \cite{istake2}). 
This is equivalent to the case that the direction of the principal vector $\bm{V}$ 
with respect to the bounded principal curvature $\kappa$ is parallel to 
the direction of $\gamma'$ along the singular curve $\gamma$. 
We now consider relationships between contactness of a singular curve with a parabolic curve 
at the cuspidal edge point and 
types of singularities of the Gauss map under the assumption that 
the singular locus $\hat{\gamma}$ is a line of curvature. 

\begin{prop}\label{curvlines}
Let $f:\Sig\to\R^3$ be a front, $\nu$ its Gauss map and $p\in\Sig$ a cuspidal edge. 
Let $\gamma$ be a singular curve passing through $p$. 
Suppose that $\kappa$ is bounded near $p$, $\kappa^{-1}(0)$ is a regular curve near $p$ 
and $\hat{\gamma}=f\circ\gamma$ is a line of curvature. 
Then 
\begin{enumerate}
\item $p$ is a fold of $\nu$ if and only if $\kappa_\nu'\neq0$ at $p$.
\item $p$ is a cusp of $\nu$ if and only if 
$\kappa_\nu'=0$, $\kappa_\nu''\neq0$ and $\kappa_s\neq0$ at $p$.
\item $p$ is a swallowtail of $\nu$ 
if and only if $\kappa_\nu'=\kappa_\nu''=0$, $\kappa_\nu'''\neq0$ and $\kappa_s\neq0$ at $p$.
\end{enumerate}
\end{prop}
\begin{proof}
Take an adapted coordinate system $(U;u,v)$ centered at $p$. 
Let us denote the principal vector with respect to $\kappa$ by $\bm{V}=V_1(u,v)\partial_u+V_2(u,v)\partial_v$. 
Since $\hat{\gamma}(u)=f(u,0)$ is a line of curvature, $V_2(u,0)=0$ (\cite[Proposition 3.2]{t2}). 
Then by the Malgrange preparation theorem (cf. \cite[Chapter IV]{gg}), 
there exists a function $W:U\to\R$ such that $V_2(u,v)=vW(u,v)$. 
Thus the principal vector $\bm{V}$ is rewritten as $\bm{V}=V_1(u,v)\partial_u+vW(u,v)\partial_v$. 
We note that $V_1(u,0)=\wtil{N}(u,0)=-\langle h(u,0),\nu_v(u,0)\rangle\neq0$ by \cite[(3.3)]{t2}, where 
$h:U\to\R^3\setminus\{0\}$ is a $C^\infty$ map satisfying $f_v=vh$. 

Under this setting, first, second and third order directional derivatives of $\kappa$ in the direction of $\bm{V}$ 
are 
\begin{align*}
\bm{V}\kappa&=V_1(\partial_u\kappa),\quad 
\bm{V}^{(2)}\kappa=V_1^2(\partial_u^2\kappa)+(\partial_uV_1)(\partial_u\kappa),\\
\quad\bm{V}^{(3)}\kappa&=
V_1(V_1^2(\partial_u^3\kappa)+3V_1(\partial_uV_1)(\partial_u^2\kappa)
+((\partial_uV_1)^2+V_1(\partial_u^2V_1))(\partial_u\kappa))
\end{align*}
at $p$, where $\partial_u^i=\partial^i/\partial u^i$ ($i=1,2,3$). 
Thus $p$ is not a ridge point if and only if $\partial_u\kappa\neq0$ at $p$, 
and $p$ is a first (resp. second) order ridge point if and only if 
$\partial_u\kappa=0$ and $\partial_u^2\kappa\neq0$ 
(resp. $\partial_u\kappa=\partial_u^2\kappa=0$ and $\partial_u^3\kappa\neq0$) at $p$. 
If $\partial_u\kappa(p)=0$, then a parabolic curve $\kappa^{-1}(0)$ through $p$ is regular 
if and only if $\partial_v\kappa(p)\neq0$. 
This is equivalent to $\kappa_s(p)\kappa_c(p)^2\neq0$ by Lemma \ref{regularparabolic}. 
Since $\kappa_c(p)\neq0$, $\kappa^{-1}(0)$ passing through $p$ is regular if and only if $\kappa_s(p)\neq0$ 
when $\partial_u\kappa(p)=0$. 

On the other hand, we have $\partial_u\kappa=\kappa_\nu'$, $\partial_u^2\kappa=\kappa_\nu''$ and 
$\partial_u^3\kappa=\kappa_\nu'''$ at $p$ 
since $\kappa=\kappa_\nu$ on the $u$-axis. 
By Theorem \ref{typegauss}, we have the conclusions.
\end{proof}
Propositions \ref{kappanu}, \ref{curvlines} and Corollary \ref{boundedness} yield the following assertion. 
\begin{cor}\label{singbound}
Under the same settings as Proposition \ref{curvlines}, 
a point $p$ is a cusp $($\/resp. swallowtail\/$)$ 
of $\nu$ if and only if $\gamma$ has $2$-point contact $($\/resp. $3$-point contact\/$)$ 
at $p$ with $\kappa^{-1}(0)$. 
Moreover, the following assertions hold:
\begin{enumerate}
\item The Gaussian curvature $K$ of $f$ is rationally bounded but not rationally continuous at $p$ 
if $p$ is a fold of $\nu$.
\item $K$ is rationally continuous at $p$ if $p$ is a cusp or a swallowtail of $\nu$.
\end{enumerate}
\end{cor}

\begin{exa}\label{exa1}
Let $f:\R^2\to\R^3$ be a $C^\infty$ map given by 
$$f(u,v)=\left(u,\frac{1}{2}u^2+\frac{1}{2}v^2,\frac{1}{3}v^3+u^4\right).$$
This gives a cuspidal edge with $S(f)=\{v=0\}$ and $\eta=\partial_v$. 
We set 
$$\nu(u,v)=\frac{1}{\sqrt{1+v^2+(-4u^3+uv)^2}}\left({-4 u^3+uv},{-v},{1}\right).$$
Then $\nu$ is the Gauss map of $f$. 
For $f$, it follows that
$$\kappa_\nu(u)=\frac{12u^2}{\sqrt{1+16u^6}(1+u^2+16u^6)},$$
$\kappa_s(0)=1\neq0$ and $\kappa_t=0$ on the $u$-axis, namely, $f(u,0)$ is a line of curvature.
Moreover, we see that $\tilde{\Lambda}(u,v)=12u^2-v$ is an identifier of singularity of $\nu$. 
Thus the parabolic curve $\kappa^{-1}(0)$ of $f$ is given by the equation $12u^2-v=0$, 
where $\kappa$ is a bounded $C^\infty$ principal curvature of $f$. 
Further, the origin is a non-degenerate singular point of $\nu$. 
By a direct computation, it follows that the singular curve has $2$-point contact at the origin with $\kappa^{-1}(0)$, 
that is, $\kappa_\nu(0)=\kappa_\nu'(0)=0$ and $\kappa_\nu''(0)\neq0$ hold.  
From Proposition \ref{curvlines}, the origin is a cusp of $\nu$ (see Figure \ref{fig:gauss}). 
Moreover, the Gaussian curvature $K$ of $f$ is rationally continuous at the origin by Corollary \ref{singbound}.
\begin{figure}[htbp]
  \begin{center}
    \begin{tabular}{c}

      % 1
      \begin{minipage}{0.33\hsize}
        \begin{center}
          \includegraphics[clip, width=3.5cm]{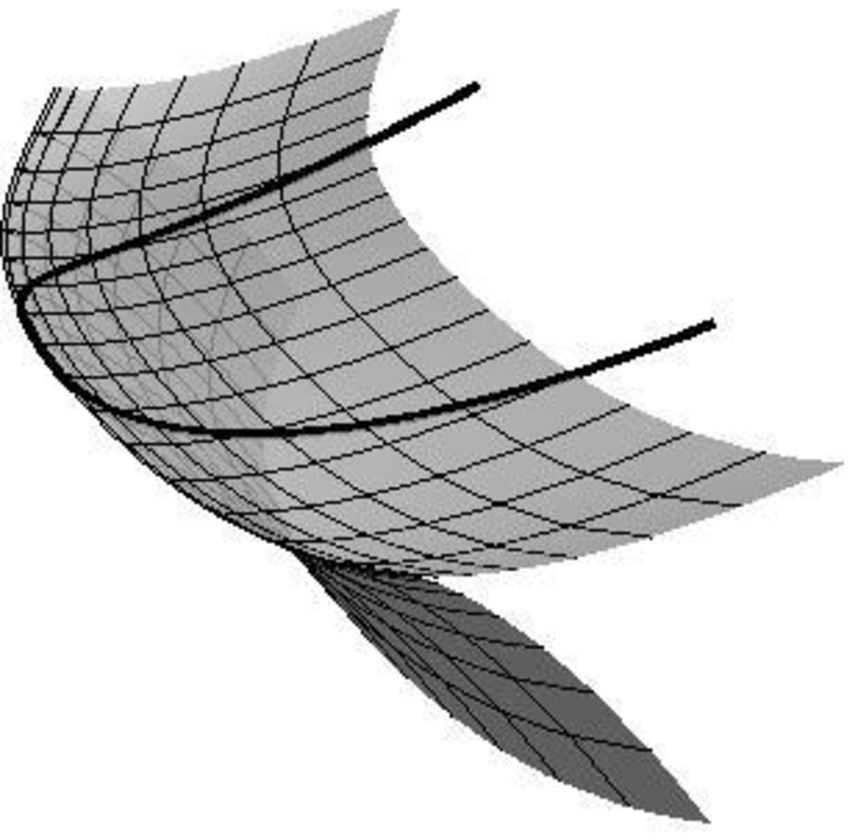}
          \hspace{1.0cm}
        \end{center}
      \end{minipage}

      % 2
      \begin{minipage}{0.33\hsize}
        \begin{center}
          \includegraphics[clip, width=3.5cm]{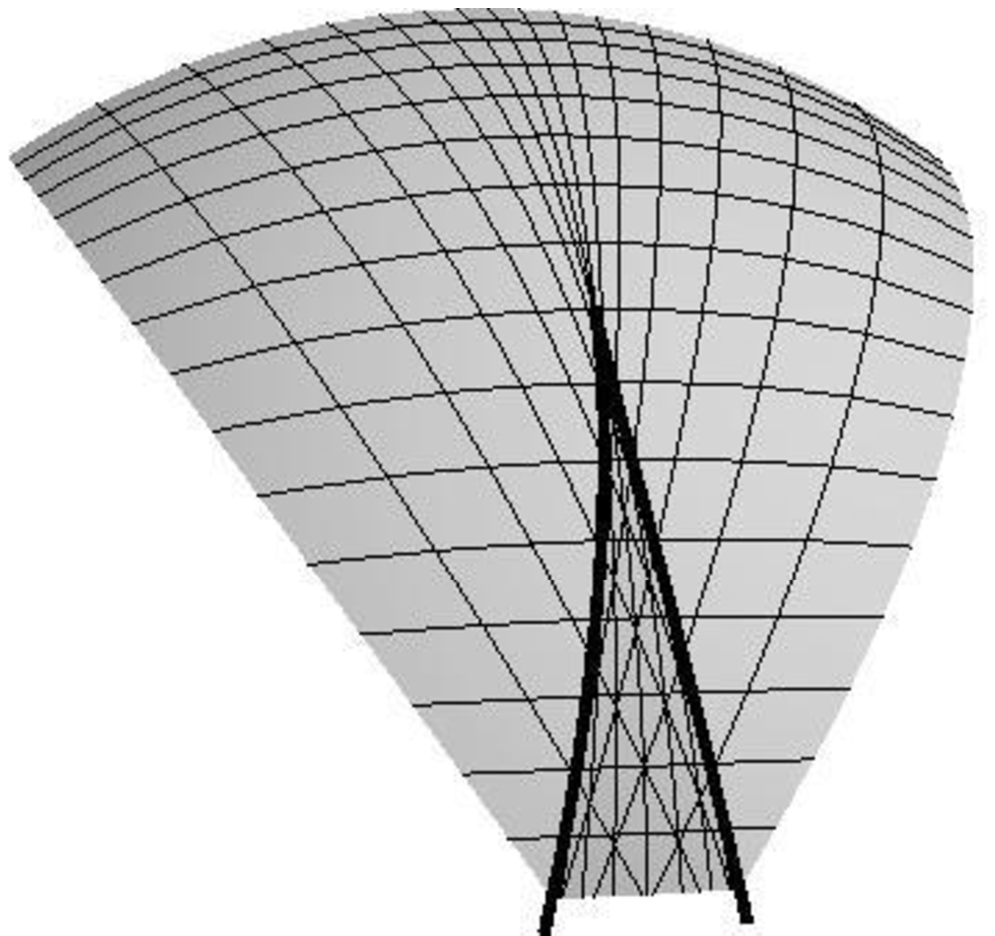}
          \hspace{1.0cm}
        \end{center}
      \end{minipage}

    \end{tabular}
    \caption{The images of cuspidal edge (left) and its Gauss map (right) in Example \ref{exa1}. 
    Thick curves on the cuspidal edge and the Gauss map are the image of the parabolic curve via $f$ and $\nu$, 
    respectively.}
    \label{fig:gauss}
  \end{center}
\end{figure}
\end{exa}

\subsection{Cuspidal edges with bounded Gaussian curvatures and their Gauss maps}
We consider a special case that the Gaussian curvature of a cuspidal edge  is bounded. 
Let $f:\Sig\to\R^3$ be a front, $p\in\Sig$ a cuspidal edge of $f$ and 
$\gamma:(-\epsilon,\epsilon)\to\Sig$ a singular curve through $p$. 
Let $\nu$ the Gauss map of $f$. 
Then it is known that {\it the Gaussian curvature $K$ of $f$ is bounded near $p$ 
if and only if the limiting normal curvature $\kappa_\nu\equiv0$ along $\gamma$} 
(see \cite[Fact 2.12]{msuy} and \cite[Theorem 3.1]{suy}). 
In this case, the set of singular points of $\nu$ coincides with $S(f)$ near $p$. 
More precisely, the bounded principal curvature $\kappa$ of $f$ defined near $p$ 
vanishes along $\gamma$. 
By Lemma \ref{regularparabolic}, $p$ is a non-degenerate singular point of $\nu$ 
if and only if $4\kappa_t(p)^2+\kappa_s(p)\kappa_c(p)^2\neq0$ holds.
\begin{prop}
Let $f:\Sig\to\R^3$ be a front, $p\in\Sig$ a cuspidal edge of $f$ and 
$\nu:\Sig\to S^2$ the Gauss map of $f$. 
Suppose that the Gaussian curvature $K$ of $f$ is bounded sufficiently small neighborhood of $p$. 
Then 
\begin{enumerate}
\item $p$ is a fold of $\nu$ if and only if $\kappa_t(p)\neq0$ and 
$4\kappa_t(p)^2+\kappa_s(p)\kappa_c(p)^2\neq0$ hold.
\item $p$ is a cusp of $\nu$ if and only if $\kappa_t(p)=0$, $\kappa_t'(p)\neq0$ and $\kappa_s(p)\neq0$ hold.
\end{enumerate}
\end{prop}
\begin{proof}
Take an adapted coordinate system $(U;u,v)$ around $p$. 
Let denote by $\kappa$ and $\bm{V}$ the bounded principal curvature on $U$ 
and the principal vector with respect to $\kappa$, respectively. 
Then by the assumption $\kappa(u,0)=\kappa_\nu(u)=0$ hold. 
Since $\partial_u\kappa(p)=\partial_u^2\kappa(p)=0$, 
we have 
$$\bm{V}\kappa(p)=V_2(p)(\partial_v\kappa(p)).$$
Since $V_2(p)\neq0$ is equivalent to $\kappa_t(p)\neq0$, 
and $\partial_v\kappa(p)\neq0$ is equivalent to 
$4\kappa_t(p)^2+\kappa_s(p)\kappa_c(p)^2\neq0$, 
we have the first assertion. 

We show the second assertion. 
We assume that $\bm{V}\kappa(p)=0$. 
Since $p$ is a non-degenerate singular point of $\nu$, $\partial_v\kappa(p)\neq0$. 
This implies that $V_2(p)=0$, namely, $\kappa_t(p)=0$.
Under these conditions, we have 
$$\bm{V}^{(2)}\kappa(p)=V_1(p)(\partial_uV_2(p))(\partial_v\kappa(p)).$$
Since $V_1(p)\neq0$ and $\partial_v\kappa(p)\neq0$, 
$p$ is a cusp of $\nu$ if and only if $\partial_uV_2(p)\neq0$. 
This is equivalent to $\kappa_t'(p)\neq0$. 
Thus we have the second assertion.
\end{proof}
We remark that similar characterizations for flat fronts in 
the hyperbolic $3$-space $H^3$ and the de Sitter $3$-space $S^3_1$ and for 
linear Weingarten fronts in $H^3$ are known (\cite{ku,st}). 
In \cite{ku}, Kokubu and Umehara studied 
global properties of linear Weingarten fronts of Bryant type 
by meromorphic representation formulae. 

\section{Extended height functions on fronts with non-degenerate singular points of the second kind}
Let $f:\Sig\to\R^3$ be a front, $\nu$ its unit normal and 
$p\in\Sig$ a non-degenerate singular point of the second kind of $f$. 
Then we take an adapted coordinate system $(U;u,v)$ centered at $p$ 
satisfying $\lambda_v=\det(h,f_v,\nu)>0$ on the $u$-axis. 
It is known that there exists a {\it strongly adapted coordinate system} $(U;u,v)$ centered at $p$ 
which is an adapted coordinate system with $\langle{f_{uv}(p),f_v(p)}\rangle=0$ (cf. \cite{msuy}). 
Thus we assume that $(U;u,v)$ is a strongly adapted coordinate system centered at $p$ 
for later calculations. 

We now define the following function:
\begin{equation}\label{height}
\phi:U\to\R,\quad \phi(u,v)=\langle{f(u,v),\bm{v}}\rangle-r,
\end{equation}
where $\bm{v}\in S^2$ is a constant vector and $r\in\R$. 
We call this function $\phi$ the {\it extended height function} on $f$ in the direction $\bm{v}$.
For other properties of height functions on regular/singular surfaces, see \cite{bgt3,bgt2,fh3,fh2,ifrt,mn,ot,t1}.
In particular, Oset Sinha and Tari \cite{ot} studied contact between cuspidal edges and planes 
using height functions, and they characterized singularities of height functions by 
geometric invariants of cuspidal edges. 
Moreover, Francisco \cite{franc} investigated functions on a swallowtail and gave a classification of functions on a swallowtail. 
\begin{lem}\label{singularity}
The extended height function $\phi$ as in \eqref{height} is singular at $p$ if $\bm{v}=\pm\nu(p)$.
\end{lem}
\begin{proof}
By direct computation, we have 
$\phi_u=\langle{f_u,\bm{v}}\rangle=\langle{vh-\epsilon(u)f_v,\bm{v}}\rangle$ and 
$\phi_v=\langle{f_v,\bm{v}}\rangle$. 
Thus we have the assertion.
\end{proof}
Taking $\bm{v}=\nu(p)$ and $r=\langle{f(p),\nu(p)}\rangle$, 
it follows that $\phi(p)=\phi_u(p)=\phi_v(p)=0$ by \eqref{height} and Lemma \ref{singularity}. 
Thus we assume that $\bm{v}=\nu(p)$ and $r=\langle{f(p),\nu(p)}\rangle$ in what follows. 
In this case, $\phi$ measures types of contact of $f$ with the limiting tangent plane at $p$, 
where the {\it limiting tangent plane} is a plane perpendicular to the unit normal vector $\nu$. 

For a function germ $\phi:(\R^2,0)\to(\R,0)$ such that $0$ is a singular point of $\phi$, 
{\it corank} of the function $\phi$ at $0$ is given by 
$\operatorname{corank}(\phi)=2-\rank\hess(\phi)$ at $0$ (cf. \cite{bg,ifrt}). 
\begin{prop}\label{hess0}
Let $f:\Sig\to\R^3$ be a front, $p$ a non-degenerate singular point of the second kind, 
$\nu$ the unit normal vector to $f$ and $\phi$ the extended height function on $f$ as in \eqref{height} 
with $\bm{v}=\nu(p)$ and $r=\langle{f(p),\nu(p)}\rangle$. 
Suppose that $\kappa$ is a bounded principal curvature of $f$ near $p$. 
Then 
\begin{itemize}
\item[${\mathrm (1)}$] $p$ is a corank $1$ singular point of $\phi$ if and only if $p$ is not a parabolic point.
\item[${\mathrm (2)}$] $p$ is a corank $2$ singular point of $\phi$ if and only if $p$ is a parabolic point.
\end{itemize}
\end{prop}

\begin{proof}
We take an adapted coordinate system $(U;u,v)$ centered at $p$. 
By Lemma \ref{singularity}, we have $\phi_u(p)=\phi_v(p)=0$. 
We consider the Hessian matrix $\hess({\phi})$ of $\phi$ at $p$. 
Since $f_{uu}=vh_u-\epsilon'f_v-\epsilon f_{uv}$ and $f_{uv}=h+vh_v-\epsilon f_{vv}$, we have 
$$\phi_{uu}(p)=\phi_{uv}(p)=0,\quad \phi_{vv}(p)=\langle{f_{vv}(p),\nu(p)}\rangle=\what{N}(p).$$
Hence the Hessian matrix of $\phi$ can be written as 
\begin{equation}\label{hessphi}
\hess({\phi}(p))=
\what{N}(p)\begin{pmatrix} 0 & 0 \\ 0 & 1\end{pmatrix}.
\end{equation}
On the other hand, $\kappa(p)=\kappa_\nu(p)=\what{N}(p)/\what{G}(p)$ holds by \eqref{principal}. 
Thus the results follow. 
\end{proof}
This proposition is a special case of \cite[Theorem 2.11]{mn}. 
In fact, Martins and Nu\~no-Ballesteros considered distance squared functions and 
height functions on a class of surfaces with corank $1$ 
singular points which contain cuspidal edges and swallowtails in \cite{mn} 
(see also \cite{franc,ot,t1,t2}). 
For a regular surface $S$, 
we note that the rank of the Hessian matrix of a height function on $S$ 
is zero at $p\in S$ if and only if $p$ is a {\it flat umbilic point} of $S$ (\cite[Proposition 2.5]{ifrt}). 

We assume that $p$ is a parabolic point in the following. 
This is the case that the Gauss map $\nu$ is singular at $p$, namely, 
the Gaussian curvature is rationally bounded at $p$. 
We now set the number $\Delta_{\phi}$ defined as 
\begin{multline}\label{delta}
\Delta_{\phi}=((\phi_{uuu})^2(\phi_{vvv})^2-6\phi_{uuu}\phi_{uuv}\phi_{uvv}\phi_{vvv}\\
-3(\phi_{uuv})^2(\phi_{uvv})^2+4(\phi_{uuv})^3\phi_{vvv}+4\phi_{uuu}(\phi_{uvv})^3)(p).
\end{multline}
It is known that a $C^\infty$ function germ $g:(\R^2,0)\to(\R,0)$ is $\mathcal{R}$-equivalent to 
$u^3+uv^2$ (resp. $u^3-uv^2$) at $0$ (see Figure \ref{fig:d4}), 
that is, $0$ is a $D_4^+$ singularity (\/resp. $D_4^-$ singularity\/) of $g$ 
 if and only if $j^2g=0$ and $\Delta_{\phi}>0$ (resp. $\Delta_{\phi}<0$), 
where $j^2g$ is the $2$-jet of $g$ at $0$ (see \cite[Lemma 3.1]{s}, see also \cite{fh}). 
By Lemma \ref{singularity} and Proposition \ref{hess0}, we see that $j^2\phi=0$ holds. 
\begin{figure}[htbp]
  \begin{center}
    \begin{tabular}{c}

      % 1
      \begin{minipage}{0.33\hsize}
        \begin{center}
          \includegraphics[clip, width=3.5cm]{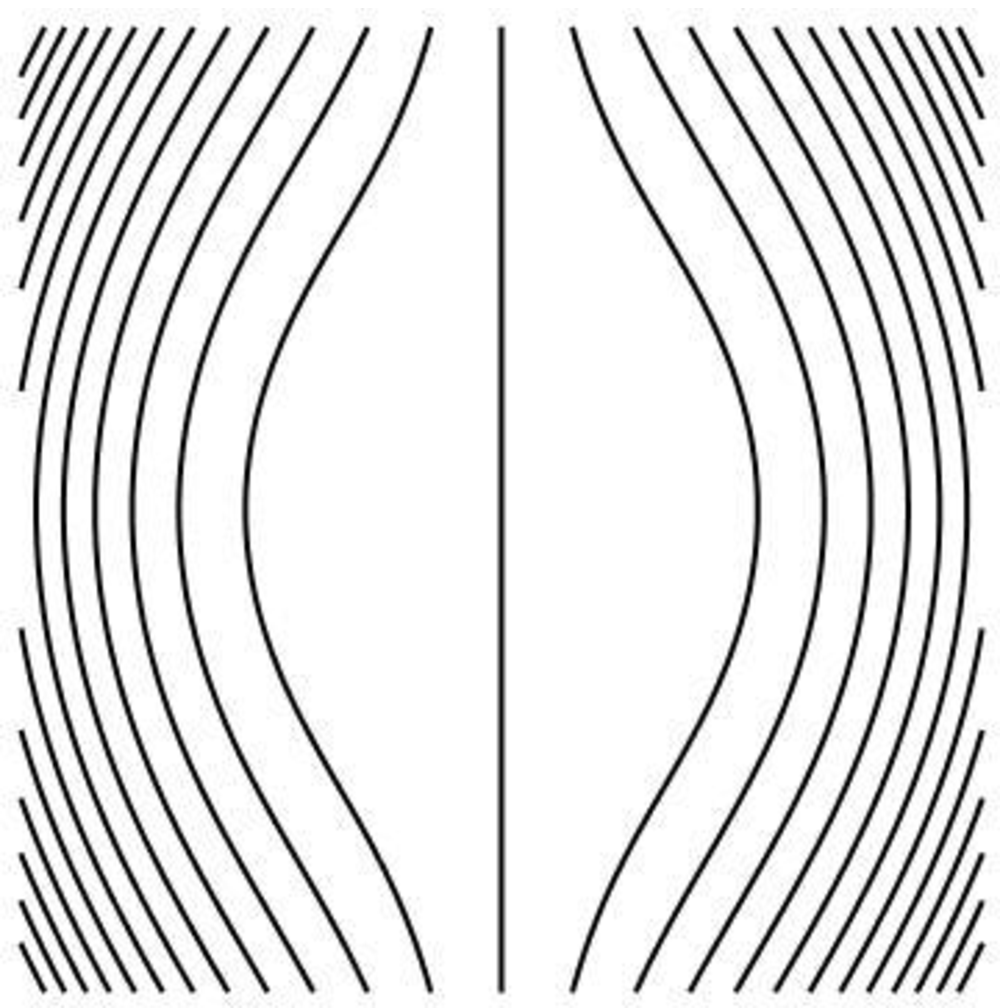}
          \hspace{1.0cm}
        \end{center}
      \end{minipage}

      % 2
      \begin{minipage}{0.33\hsize}
        \begin{center}
          \includegraphics[clip, width=3.5cm]{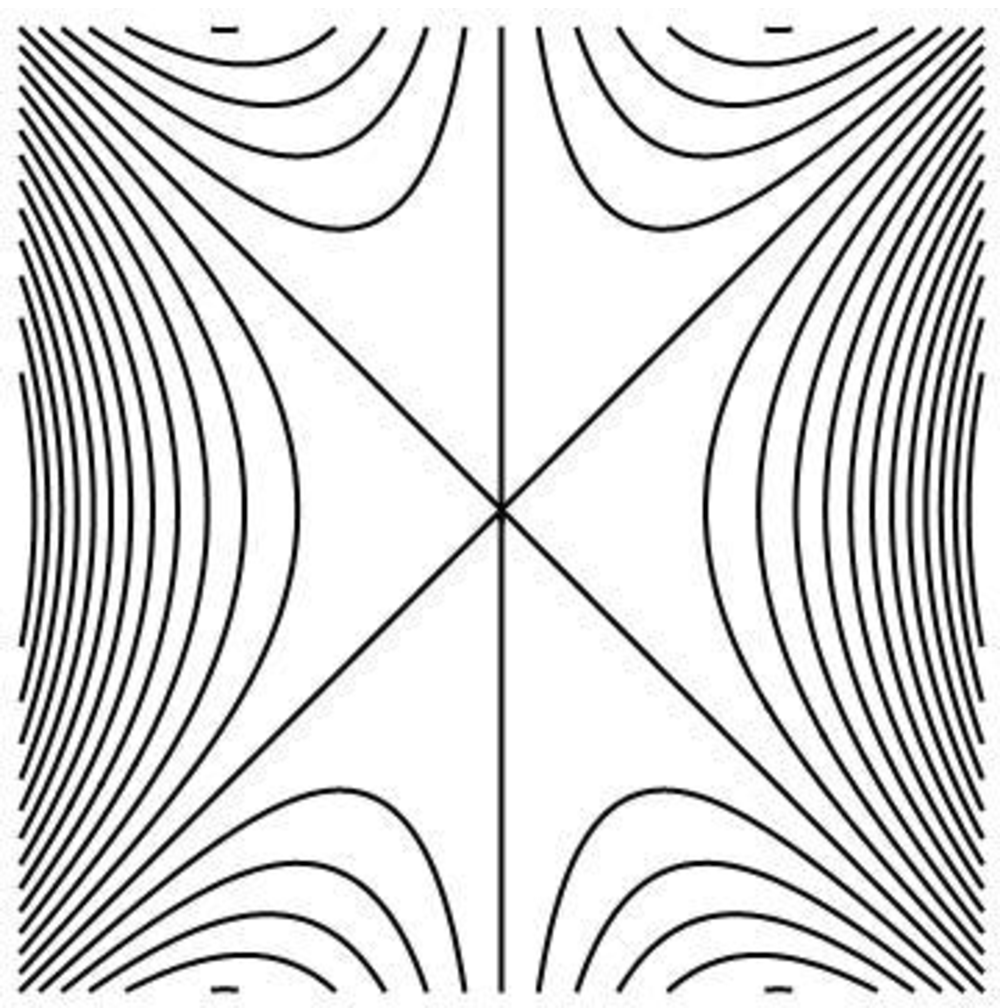}
          \hspace{1.0cm}
        \end{center}
      \end{minipage}

    \end{tabular}
    \caption{Level sets of functions $u^3+uv^2$ (left) and $u^3-uv^2$ (right).}
    \label{fig:d4}
  \end{center}
\end{figure}

\begin{thm}\label{D4phi}
Under the above conditions, the function $\phi$ as in \eqref{height} with 
$\bm{v}=\nu(p)$ and $r=\langle{f(p),\nu(p)}\rangle$
has a $D_4$ singularity at $p$ 
if and only if $p$ is a parabolic point and not a ridge point of a front $f$.
\end{thm}
\begin{proof}
Take a strongly adapted coordinate system $(U;u,v)$ centered at $p$. 
Since $f$ is a front at $p$, $\what{L}(p)\neq0$ holds. 
Thus we may assume that $\what{L}(p)>0$. 
By Proposition \ref{hess0}, $j^2\phi=0$ if and only if $p$ is a parabolic point. 
Hence we suppose that $p$ is a parabolic point, namely $\what{N}(p)=0$. 

First, we consider the number $\Delta_{\phi}$ as in \eqref{delta}. 
Since $f_{uuu}=vh_{uuu}-\epsilon'' f_v-2\epsilon'f_{uv}-\epsilon f_{uuv}$ and $f_{uv}=h+vh_v-\epsilon f_{vv}$, 
we have $\phi_{uuu}(p)=0$. 
Thus it holds that 
$$\Delta_{\phi}=(\phi_{uuv})^2(p)(4\phi_{uuv}\phi_{vvv}-3(\phi_{uvv})^2)(p).$$
By direct computations, we see that 
$$\phi_{uuv}=\what{L},\quad\phi_{uvv}=2\what{M},\quad\phi_{vvv}
=\what{N}_v-\frac{\what{M}}{\what{L}}(\what{N}_u-2\what{M})$$
at $p$. 
Therefore we have 
\begin{equation}\label{delta2}
\Delta_{\phi}=4\what{L}(p)^2(\what{L}(p)\what{N}_v(p)-\what{M}(p)(\what{N}_u(p)+\what{M}(p))).
\end{equation}

Next, we consider the condition of ridge points. 
Under the above settings, $\kappa$ is given by \eqref{principal}. 
The differentials $\partial_u\kappa=\partial\kappa/\partial u$ and $\partial_v\kappa=\partial\kappa/\partial v$ are 
$$\partial_u\kappa(p)=\frac{\what{N}_u(p)}{\what{G}(p)},\quad
\partial_v\kappa(p)=\frac{\what{L}(p)\what{N}_v(p)-\what{M}(p)^2}{\what{G}(p)\what{L}(p)}.$$
The principal vector $\bm{V}$ relative to $\kappa$ is $\bm{V}=(-\what{M},\what{L})$ at $p$. 
Thus the directional derivative $\bm{V}\kappa$ of $\kappa$ in the direction $\bm{V}$ at $p$ is given by 
\begin{equation}\label{ridge}
\bm{V}\kappa(p)=\frac{1}{\what{G}(p)}(\what{L}(p)\what{N}_v(p)-\what{M}(p)(\what{N}_u(p)+\what{M}(p))).
\end{equation}
Comparing \eqref{delta2} and \eqref{ridge}, we get the conclusion.
\end{proof}
%We remark that the assumption that $f$ is a front at $p$ is crucial for Theorem \ref{D4phi}. 
%If $f$ is a frontal but not a front at $p$, $\phi_{uuv}(p)$ vanishes since $\mu_c(p)=0$, namely, $\what{L}(p)=0$. 
%This means that $\Delta_{\phi}$ as in \eqref{delta} vanishes automatically. 

By Theorems \ref{typegauss} and \ref{D4phi}, we have the following assertion.
\begin{cor}
Under the same conditions as Theorem \ref{D4phi}, 
the function $\phi$ as in \eqref{height} with $\bm{v}=\nu(p)$ and $r=\langle f(p),\nu(p)\rangle$ 
has a $D_4$ singularity at $p$ if and only if the Gauss map $\nu$ has a fold singularity at $p$.
\end{cor}

%\proof[Acknowledgements]
%The author is grateful to 
%Professor Kentaro Saji for fruitful discussions and constant encouragement.

\end{document}